\documentclass[11pt,letterpaper]{article}
\usepackage{fullpage}
\usepackage{amsthm}

\usepackage{authblk}
\usepackage{amssymb}
\usepackage[compact]{titlesec}
\usepackage{amsmath}
\usepackage{nicefrac}
\usepackage{mathdots,amsthm,url}
\usepackage{ifpdf,color}
\usepackage{graphicx}
\usepackage{subfigure}
\usepackage{enumerate}
\usepackage{mathtools}
\usepackage{thm-restate}
\usepackage[margin=1.0in]{geometry}
\usepackage[normalem]{ulem}

\usepackage[linesnumbered, ruled, lined, boxed, noend]{algorithm2e}
\SetKwInput{KwInput}{Input}

\usepackage[colorinlistoftodos]{todonotes}
\usepackage[colorlinks=true, allcolors=blue]{hyperref}
\usepackage{cleveref}

\newtheorem{theorem}{Theorem}[section]
\newtheorem{claim}[theorem]{Claim}
\newtheorem{proposition}[theorem]{Proposition}
\newtheorem{lemma}[theorem]{Lemma}
\newtheorem{corollary}[theorem]{Corollary}

\theoremstyle{definition}

\newtheorem{definition}[theorem]{Definition}

\SetKwInOut{Input}{Input}
\SetKwInOut{Output}{Output}

\newcommand{\targetaccuracy}{\zeta}
\newcommand{\alphaxtau}{\alpha}

\newcommand{\betaxtau}{\beta}

\newcommand{\ceps}{\hat c}
\newcommand{\paraK}{\mathcal{C}}
\newcommand{\paraBarK}{\overline{\mathcal{C}}}
\newcommand{\Kc}{h_{\scriptscriptstyle{\!R}}}

\newcommand{\oriA}{A}
\newcommand{\hide}[1]{}
\newcommand{\Ft}[1]{F_{#1}}
\newcommand{\Ftt}[1]{\tilde F_{#1}}
\newcommand{\Ftopt}[1]{F_{#1}^\star}
\newcommand{\Fctau}{F_{\tau}}
\newcommand{\deltafeas}{\delta^\feas}
\newcommand{\deltaopt}{\delta^\opt}
\newcommand{\deltast}{\deltafeas}

\newcommand{\Fcepstau}{F_{\tau}}
\newcommand{\Fcepstauopt}{F_{\tau}^\star}

\newcommand{\LP}{\mathrm{LP}}
\newcommand{\Dual}{\mathrm{Dual}}
\newcommand{\optv}{\Phi}
\newcommand{\bJ}{J}
\newcommand{\ee}{\mathrm{e}}
\newcommand{\EE}{\mathcal F}

\newcommand{\cX}{\mathcal{X}}
\newcommand{\bO}{\mathbf{0}}

\newcommand{\R}{\mathbb{R}}
\newcommand{\Z}{\mathbb{Z}}

\newcommand{\cP}{\mathcal{P}}

\newcommand{\pr}[2]{\left\langle #1,#2\right\rangle}

\newcommand{\supp}{\mathrm{supp}}

\newcommand{\uJtwo}{u_{\scriptscriptstyle{\!J_2}}}
\newcommand{\uN}{u_{\scriptscriptstyle{\!N}}}

\newcommand{\AJone}{A_{\scriptscriptstyle{\!J_1}}}
\newcommand{\AJtwo}{A_{\scriptscriptstyle{\!J_2}}}
\newcommand{\AJ}{A_{\scriptscriptstyle{\!J}}}
\newcommand{\AN}{A_{\scriptscriptstyle{\!N}}}
\newcommand{\bbar}{\bar b}

\newcommand{\cJtwo}{c_{\scriptscriptstyle{\!J_2}}}
\newcommand{\cN}{c_{\scriptscriptstyle{\!N}}}
\newcommand{\xbar}{\bar x}
\newcommand{\xJone}{x_{\scriptscriptstyle{\!J_1}}}
\newcommand{\xJtwo}{x_{\scriptscriptstyle{\!J_2}}}
\newcommand{\xJ}{x_{\scriptscriptstyle{\!J}}}
\newcommand{\xN}{x_{\scriptscriptstyle{\!N}}}

\newcommand{\zbar}{\bar{z}}

\newcommand{\RFGM}{\texttt{R-FGM}}
\newcommand{\Checkalg}{\texttt{Dual-Certificate}}
\newcommand{\Feasiblealg}{\texttt{Feasible}}
\newcommand{\getPair}{\texttt{GetPrimal\-DualPair}}
\newcommand{\solveLP}{\texttt{SolveLP}}
\newcommand{\opt}{\mathrm{opt}}
\newcommand{\feas}{\mathrm{feas}}
\newcommand{\new}{\mathrm{new}}
\newcommand{\gain}{\lambda}
\newcommand{\outindex}{\mathrm{out}}

\newcommand{\la}{\leftarrow}

\newcommand\numberthis{\addtocounter{equation}{1}\tag{\theequation}}

\newcommand\blfootnote[1]{%
  \begingroup
  \renewcommand\thefootnote{}\footnote{#1}%
  \addtocounter{footnote}{-1}%
  \endgroup
}

\title{A First Order Method for Linear Programming\\ Parameterized by Circuit Imbalance\blfootnote{This work was supported by the European Research Council (ERC) under the European Union's Horizon 2020 research and innovation programme (grant agreement no.~ScaleOpt--757481; for C.~Hertrich additionally via grant agreement no.~ForEFront--615640). Y. Tao also acknowledges Grant 2023110522 from SUFE, National Key R\&D Program of China (2023YFA1009500), NSFC grant 61932002. Part of the work was done while L.~V\'egh was visiting the Corvinus Institute for Advanced Studies, Corvinus University, Budapest, Hungary, and while C.~Hertrich was affiliated with London School of Economics, UK, and with Goethe-Universit\"at Frankfurt, Germany.}}
\author[1]{Richard Cole}
\author[2]{Christoph Hertrich}
\author[3]{Yixin Tao} 
\author[4]{L{\'{a}}szl{\'{o}} A. V{\'{e}}gh}
\affil[1]{Courant Institute, New York University,  USA}
\affil[2]{Universit\'e Libre de Bruxelles, Belgium} 
\affil[3]{ITCS, Key Laboratory of Interdisciplinary Research of Computation and Economics\\ Shanghai University of Finance and Economics, China}
\affil[4]{London School of
	Economics, UK, 
	and Corvinus Institute for Advanced Studies, Corvinus University, Budapest, Hungary}

\begin{document}

\begin{titlepage}
\clearpage
  \maketitle
\thispagestyle{empty}
\begin{abstract}
Various first order approaches have been proposed in the literature to solve Linear Programming (LP) problems, recently leading to practically efficient solvers for large-scale LPs. From a theoretical perspective, linear convergence rates have been established for first order LP algorithms, despite the fact that the underlying  formulations are not strongly convex.
However, the convergence rate typically depends on 
the Hoffman constant of a large matrix that contains the constraint matrix, as well as the right hand side, cost, and capacity vectors.

We introduce a first order approach for LP optimization with a convergence rate depending polynomially on the circuit imbalance measure, which is a geometric parameter of the constraint matrix, and depending logarithmically on the right hand side, capacity, and cost vectors. This provides much stronger convergence guarantees. For example, if the constraint matrix is totally unimodular, we obtain polynomial-time algorithms, whereas the convergence guarantees for approaches based on primal-dual formulations may have arbitrarily slow convergence rates for this class.
Our approach is based on a fast gradient method due to Necoara, Nesterov, and Glineur (Math. Prog. 2019); this algorithm is called repeatedly in a framework that gradually fixes variables to the boundary. This technique is based on a new approximate version of 
Tardos's method, that was used to obtain a strongly polynomial algorithm for combinatorial LPs (Oper. Res. 1986).

\end{abstract}
\end{titlepage}
 \setcounter{page}{1}

\section{Introduction}
In this paper, we develop new first order algorithms for approximately solving the linear program
\begin{equation*}\label{eq:LP}\tag{LP$(A,b,c,u)$}
\begin{aligned}
\min\ & \pr{c}{x}\\
Ax&=b\, ,\\
\bO\le x&\le u\, ,
\end{aligned}
\end{equation*}
where $A\in \R^{m\times n}$, $b\in \R^m$, $c,u\in\R^n$. We assume that $m\le n$. We use the notation $[\bO,u]=\{x\in\R^n\mid \bO\le x\le u\}$, and denote the feasible region as $\cP_{A,b,u}:=\{x\in\R^n\mid\, Ax=b\, ,\, x\in [\bO,u]\}$.

Linear programming  (LP) is one of the most fundamental optimization problems with an immense range of applications in applied mathematics, operations research, computer science, and more. While Dantzig's Simplex method works well in practice, its running time may be exponential in the worst case. Breakthrough results in the 1970s and 1980s led to the development of the first polynomial time algorithms, the ellipsoid method \cite{Khachiyan79} and interior point methods (IPMs) \cite{Karmarkar84}.  The Simplex method
was one of the earliest computations {implemented}  on a computer, and there are highly efficient LP solvers available, based on the Simplex and interior point methods.

 Linear programming can also be seen as a special case of more general optimization models: it can be captured by various convex programs, saddle point problems, and linear complementarity problems. {These connections led to} the development of new LP algorithms {being}  an important driving force in the development of optimization theory.

In this paper, we focus on first order methods (FOMs) for LP. 
The benefit of FOMs is cheap iteration complexity and efficient  implementability for large-scale problems. In contrast to IPMs, they do not require  careful initialization. FOMs are prevalent in optimization and machine learning, but they are not an obvious choice for LP for two reasons. First, the standard formulation has a complicated polyhedral feasible region, and therefore standard techniques are not directly applicable. Second, FOMs usually do not lead to polynomial running time guarantees: this is in contrast with IPMs that are polynomial and also efficient in practice.

Nevertheless, FOMs turn out to be practically efficient for large-scale LPs. In a recent paper Applegate et al.~\cite{applegate2021practical} use a restarted primal-dual hybrid gradient (PDHG) method based on a saddle point formulation. Their implementation outperforms the state-of-the art commercial Simplex and IPM solvers on standard benchmark instances, and is able to find high accuracy solutions to large-scale PageRank instances.

The number of iterations needed to find an $\varepsilon$-approximate solution in standard FOMs  is typically  $O(1/\varepsilon)$ or $O(1/\sqrt{\varepsilon})$. However, strong convexity properties can yield  \emph{linear convergence}, i.e., an $O(\log(1/\sqrt{\varepsilon}))$ dependence. No strongly convex formulation is known to capture LP. Despite this, there is a long line of work on FOMs that achieve linear convergence guarantees for LP, starting from Eckstein and Bertsekas's alternating direction method from 1990 \cite{eckstein1990alternating}, followed by a variety of {other} techniques, e.g., \cite{applegate2021practical,applegate2023faster,gilpin2012first,hinder2023,necoara2019,wang2017new,yang2018rsg}.

Before discussing these approaches, let us specify the notion of approximate solutions.
 By a \emph{$\delta$-feasible} solution, we mean an $x\in[\bO,u]$ with $\|Ax-b\|_1\le\delta\|A\|_1$. If \ref{eq:LP} is feasible, we let $\optv(A,b,c,u)$ denote the optimum value. A \emph{$\delta$-optimal} solution satisfies $\pr{c}{x}\le \optv(A,b,c,u)+\delta\|c\|_\infty$. Our goal will be to find a $\delta$-feasible and $\delta$-optimal solution for a required accuracy $\delta>0$. Different papers may use different norms and normalizations in their accuracy requirement, but these can be easily translated to each other.

The above mentioned works are able to find $\delta$-feasible and $\delta$-optimal solutions in running times that depend polynomially on $\log(1/\delta)$, $n$,  and $C(A,b,c,u)$, a constant depending on the problem input.
In particular, Applegate, Hinder, Lu, and Lubin~\cite{applegate2023faster} give a running time bound $O(C\log(1/\delta))$ for restarted PDHG, where $C$ is the Hoffman-constant associated with the primal-dual embedding of the LP. 
Recently, for the case when $A$ is a totally unimodular matrix and there are no upper bounds $u$, Hinder~\cite{hinder2023}  bounded the running time of restarted PDHG as
$O(H n^{2.5}\sqrt{\mathrm{nzz}(A)}\log(Hm/\delta))$, where $\mathrm{nzz}(A)$ is the number of  nonzero entries of $A$, and $b,c$ are integer vectors with $\|b\|_\infty,\|c\|_\infty\le H$.

However, the constants involved in the running time bound  are typically not polynomial in the binary encoding length of the input. In this paper, we give the first FOM-based algorithm with polynomial dependence on $\log(1/\delta)$, $n$, a constant $\bar\kappa(\cX_A)$, and $\log\|b\|$, $\log\|c\|$, and $\log\|u\|$, as stated in \Cref{thm::main-result} below. The constant $\bar\kappa(\cX_A)$ is the \emph{max circuit imbalance measure} defined and discussed below. In particular, it is upper bounded by the maximal subdeterminant $\Delta(A)$, but it is often much smaller than $\Delta(A)$. We have $\bar\kappa(\cX_A)=1$ if $A$ is a totally unimodular matrix. Note that the running time depends polynomially on the logarithms of the capacity and cost vectors. In contrast, the bound in \cite{hinder2023} only applies for the totally unimodular case and the running time is linear in \mbox{$\|b\|_\infty+\|c\|_\infty$}.

We also note that one may always rescale the matrix $A$ to have $\|A\|_1=1$; however, such a rescaling may change $\bar\kappa(\cX_A)$.

\begin{theorem}\label{thm::main-result}
    There is an FOM-based algorithm for \ref{eq:LP} that obtains a solution $x$ that is $\delta$-feasible and $\delta$-optimal, or concludes that no feasible solution exists, and  whose runtime is dominated by the cost of performing 
$O\big( n^{1.5} m^2  \|A\|_1^2 \cdot\bar\kappa^3(\cX_{A})$ $\log^3 \left((\|u\|_1+\|b\|_1) n m \cdot \kappa(\cX_A) \|A\|_1 /\delta\right)\big)$ gradient descent updates, where we assume $\|A\|_1\ge 1$.
Additionally, our  algorithm returns a dual solution certifying approximate optimality of the solution in $O\big(m \|A\|_2 \cdot\bar\kappa(\cX_A)\cdot \log(n \|c\|_1 / \delta)\big)$ gradient descent updates.
\end{theorem}

\paragraph{Hoffman bounds and quadratic function growth.}
The main underlying tool for proving linear convergence bounds is Hoffman-proximity theory, introduced by Hoffman in 1952 \cite{Hoffman52}. 
Let $A\in\R^{m\times n}$, let $\|.\|_\alpha$ be a norm in $\R^m$ and $\|.\|_\beta$ be a norm in $\R^n$. Then there exists a constant $\theta_{\alpha,\beta}(A)$ such that for any $x\in[\bO,u]$, and any $b\in\R^m$, whenever $\cP_{A,b,u}$ is nonempty, there exists an $\bar x\in\cP_{A,b,u}$ such that
\[
\|\bar x-x\|_\beta\le \theta_{\alpha,\beta}(A)\|Ax-b\|_\alpha\, .
\]
To see how such bounds can lead to linear convergence, let us first focus on finding a feasible solution in $\cP_{A,b,u}$. This can be formulated as a convex quadratic minimization problem:
\begin{equation}\label{eq:feasibility}
\begin{aligned}
\min \tfrac{1}{2}\|Ax-b\|^2\quad \mbox{s.t.}\quad x\in[\bO,u]\, .
\end{aligned}
\end{equation}
This is a smooth objective function, but not strongly convex. Nevertheless, Hoffman-proximity guarantees that for any $x\in [\bO,u]$ where $f(x):=\frac{1}{2}\|Ax-b\|^2$ is close to the optimum value, there exists \emph{some} optimal solution $\bar x$ nearby. Nesterov, Necoara, and Glineur~\cite{necoara2019} introduce various relaxations of strong convexity, including the notion of $\mu_f$-quadratic growth (Definition~\ref{def:quadratic-growth}), and show that these weaker properties suffice for linear convergence.

Hence, \cite{necoara2019} implies that a $\delta$-feasible LP solution can be found by a Fast Gradient Method with Restart (R-FGM) in $O(\|A\|_2 \cdot\theta_{2,2}(A)\log(m\|b\|_1/\delta))$ iterations. 
 {If $A$ is}  a totally unimodular (TU) matrix, the dependence is given by $\theta_{2,2}(A)\le m$ (see Lemma~\ref{lem:hoffman-kappa}).

{To solve}  \ref{eq:LP}, \cite{necoara2019} in effect uses the standard reduction from optimization to feasibility by writing the primal and dual systems together. By strong duality, if \ref{eq:LP} is feasible and bounded, then $x$ is a primal and $(\pi,w^+,w^-)$ {is} a dual optimal solution if and only if
\begin{equation}\label{eq:self-dual}
Ax=b\,  ,\quad A^\top \pi + w^--w^+=c\, , \quad \pr{c}{x}-\pr{b}{\pi}+\pr{u}{w^+}=0\,  ,\quad  x,w^-,w^+\ge 0\,. 
\end{equation}
We can use R-FGM for this larger feasibility problem. However, the constraint matrix $M$ now also includes the vectors $c$, $b$, and $u$. In particular, while $\theta_{2,2}(A)$ is small for a TU matrix, $\theta_{2,2}(M)$ may be unbounded, as shown in Section~\ref{sec::examples}.

Other previous works obtain linear convergence bounds using different approaches, but share the above characteristics: their running time includes a constant term $C(A,b,c,u)$. For example, \cite{eckstein1990alternating} and \cite{wang2017new} use an alternating direction method based on an augmented Lagrangian, and \cite{applegate2023faster} and \cite{hinder2023} use restart PDHG. The convergence bounds depend not only on the Hoffman-constant of the system, but also linearly on the maximum possible norm of the primal and dual iterates seen during the algorithm. 

\subsection{Our approach}

We present an algorithm in the FOM family with polynomial dependence on 
$\log(1/\delta)$, $n$, $m$,  $\log\|u\|$, $\log \|b\|$, $\log \|c\|$ and a constant $C(A)$ only dependent on $A$. Our algorithm repeatedly calls R-FGM, described in \cite{necoara2019}, on a potential function of the form
\begin{equation}\label{eq:F-tau-intro}
\Ft{\tau}(x) := \frac{1}{2}(\max\{0, \langle \ceps, x \rangle  - \tau\})^2 + \frac{1}{2\|A\|_1^2} \|Ax - b\|_2^2\, , 
\end{equation}
for a suitably chosen parameter $\tau\in \R$, and a modified cost function $\ceps$. If we use $\ceps=c/\|c\|_\infty$, and $\tau$ is slightly below the optimum value, then one can show that a near-minimizer $x$ of $\Ft{\tau}(x)$ is a near optimal primal solution to the original LP, and moreover, we can use the gradient $\nabla\Ft{\tau}(x)$ to construct a near-optimal dual solution to the problem.

Thus, one could find a $\delta$-approximate and $\delta$-optimal solution to \ref{eq:LP}  with $\log(1/\delta)$ dependence by doing a binary search over the possible values of $\tau$, and running R-FGM for each guess. This already improves on the parameter dependence, however, it still involves a constant $C(A,c)$. One can formulate the minimization of $\Ft{\tau}(x)$ in the form \eqref{eq:feasibility};  the constraint matrix also includes the vector $\ceps$. The resulting Hoffman constant can be arbitrarily worse than the one for the original system.

To overcome this issue, we instead define $\ceps$ as {an}  $\varepsilon$-discretization of $c/\|c\|_\infty$. We show that the Hoffman constant remains bounded in terms of the Hoffman constant of the feasibility system and {a} suitably chosen $\varepsilon>0$.
     Now, for the appropriate choice of $\tau$, a near-minimizer of $\Ft{\tau}(x)$ only gives a crude approximation to the original LP: the error depends on the discretization parameter $\varepsilon$, and to keep the Hoffman constant under control we cannot choose $\varepsilon$ very small. Nonetheless, the dual solution obtained from the gradient contains valuable information. For certain indices $i\in N$, using primal-dual slackness, one can conclude $x^*_i\approx 0$ or $x_i^*\approx u_i$ for an optimal solution $x^*$ to the original LP. We fix all such $x_i$ to 0 or $u_i$, respectively, and recurse. Even if we not find any such $x_i$, we make progress by replacing our cost function by an equivalent reduced cost with the $\ell_\infty$ norm decreasing by at least a factor two.

{To summarize:}  our overall algorithm has an outer loop that gradually fixes the variables to the upper and lower bounds, and repeatedly replaces the cost by a reduced cost. In the inner loop, we call R-FGM in a binary search framework {that guesses} the parameter $\tau$. We note that while R-FGM is run on a number of systems, the total number of these systems is logarithmically bounded in $1/\delta$ and the input parameters. Moreover, besides the first order updates, we only perform simple arithmetic operations: based on the gradient, we eliminate a subset of variables and shift the cost function. On a high level, our algorithm is a repeatedly applied FOM, where after each run, we `zoom in' to a `critical' part of the problem based on what we learned from the previous iteration.

\paragraph{Circuit imbalance measures and proximity}
The key parameters for our algorithm are \emph{circuit imbalance measures}. For a linear space $W\subseteq \R^n$, an elementary vector is a support minimal nonzero vector in $W$. The \emph{(fractional) circuit imbalance measure} $\kappa(W)$ is the largest ratio between the absolute values of two entries of an elementary vector.
If $W$ is a rational space, then every elementary vector can be rescaled to have integer entries; and the \emph{max circuit imbalance measure} $\bar\kappa(W)$ is the smallest integer $k$ such that all elementary vectors can be scaled to have integer entries between $-k$ and $k$. For a matrix $A$, we also use $\kappa(A)=\kappa(\ker(A))$ and $\bar\kappa(A)=\bar\kappa(\ker(A))$. We give a more detailed introduction to these measures in Section~\ref{sec:circuits}. We define the subspace $\cX_A=\ker(A|-I_m)$; thus, $(v,-Av)\in \cX_A$ for any $v\in\R^n$.

Circuit imbalances play two roles in our algorithm. First, they are used to bound the number of iterations of R-FGM.
The circuit imbalance measure of $\cX_A$ gives the  Hoffman-proximity bound $\theta_{1,\infty}(A)\leq\kappa(\cX_A)$ (see Lemma~\ref{lem:hoffman-kappa}). To bound the number of iterations in R-FGM, we need a Hoffman bound---equivalently, a circuit imbalance bound---for a matrix representation of \eqref{eq:F-tau-intro}; this matrix $B$ arises by adding an additional row containing $\ceps$ to the matrix $A$. For this, we need to use the max circuit imbalance measure $\bar\kappa(\cX_A)$; we can show $\kappa(\cX_B)\le 2m\cdot\bar\kappa^2(\cX_A)/\varepsilon$. Remember that the $\varepsilon$ comes from~$\ceps$, which is an $\varepsilon$-discretization of~$c/\|c\|_\infty$.

The second role of $\kappa(\cX_A)$ is for the variable fixing argument in the outer loop of the algorithm. Recall that the inner loop returns a near optimal primal solution with respect to the rounded cost, as well as a near optimal dual solution derived from the gradient of the potential function. We would like to infer that variables with a large positive or negative dual slack can be rounded to the lower or upper bounds. To make such an inference, the rounding accuracy $\varepsilon$ needs to be calibrated to $\kappa(\cX_A)$. The larger $\kappa(\cX_A)$ is, the more refined the rounding needed to obtain such guarantees.

\paragraph{Guessing the condition numbers.} Our algorithm requires explicit bounds on the circuit imbalance measures both in the inner and outer loops. However, the circuit imbalance measures cannot be approximated even within an exponential factor unless $P=NP$ (see Section~\ref{sec:circuit-hard}). Similar issues arise in several algorithms that rely on condition numbers. In particular, R-FGM in \cite{necoara2019} explicitly requires a bound on the Hoffman constant to determine the step-length; it does not address how such a bound could be obtained.\footnote{However, the same paper shows that standard projected gradient also converges linearly albeit at a slower rate, and does not require knowing this constant.}

We circumvent this problem by using a simple doubling guessing procedure: starting with the guess $\hat \kappa=1$, we run the algorithm. Either it succeeds, or otherwise we restart after doubling the guess $\hat\kappa$. The asymptotic running time is the same as when the circuit imbalance value is known. The only nontrivial issue is checking whether we succeeded; this can be done by running a final dual feasibility algorithm.

We note that the algorithm may succeed even if $\hat\kappa$ is much better than the actual $\bar\kappa(\cX_A)$ value. {In fact,} the guessing procedure is  a natural heuristic. We initially start with a crude discretization strategy to guess variable fixings from the outputs of R-FGM. In {the event} this leads to an infeasible or suboptimal solution, we restart after increasing the accuracy.

\subsection{Related work}

We recall that an LP algorithm is strongly polynomial if it only uses basic arithmetic operations ($+,-,\times,/$) and comparisons, and the number of such operations is polynomial in the number of variables and constraints. Further, the algorithm must be in PSPACE, that is, the size of the numbers appearing in the computation must remain bounded in the input size. The existence of a strongly polynomial algorithm for LP is on Smale's list of main challenges for 21st century mathematics \cite{Smale98}.

The variable fixing idea in our algorithm traces its roots  to Tardos's strongly polynomial algorithm for minimum-cost circulations. The same idea was extended by Tardos  \cite{Tardos86} to obtain {a} poly$(n,\log\Delta(A))$ time algorithm for finding an exact solution to \ref{eq:LP} for an integer constraint matrix $A\in\mathbb{Z}^{m\times n}$ with largest subdeterminant $\Delta(A)$. This running time bound is strongly polynomial for `combinatorial LPs', that is, LPs with all entries being integers of absolute value poly$(n)$.

We note that $\kappa(A)\le \kappa(\cX_A)\le\bar\kappa(\cX_A)\le\Delta(A)$ for an integer matrix $A$. 
 Dadush et al.~\cite{DadushNV20} strengthened Tardos' result by replacing $\Delta(A)$ by $\kappa(A)$, and removing all integrality-based arguments, and obtained a poly$(n,\log\kappa(A))$ running time bound. The algorithm is of black-box nature, and can use any LP solver; an exact optimal solution can be found by running $nm$ LP-solvers to accuracy $\delta=1/\mathrm{poly}(n,\log\kappa(A))$.

Our algorithm uses variable fixing in a different manner, giving a robust extension to the approximate setting.
 Our end goal is not  an exact optimal solution, but {rather} an approximate one (that could serve as an input for  the black-box algorithm \cite{DadushNV20}). The approximate solution obtained from the FOM in the inner loop has weaker guarantees.  Tardos \cite{Tardos86} also uses subproblems with a similarly rounded cost function, but requires exact feasibility, which cannot be obtained from an FOM.

For this reason, we obtain weaker guarantees, and may fix to $0$ variables that are small but positive in all optimal solutions. However, this is  acceptable if we are only aiming for an approximate solution. On the positive side, we only need a logarithmic number of executions of the outer loop, in contrast to $nm$ in \cite{DadushNV20,Tardos86}. This is because for us it is already sufficient progress to decrease the norm of the reduced cost, even if we cannot fix any variables.

A  poly$(n,\log\kappa(A))$ running time for LP can also be achieved by a special class of `combinatorial' interior point methods, called \emph{Layered Least Squares (LLS) IPMs}. This class was introduced by Vavasis and Ye \cite{Vavasis1996}. The parameter dependence was on the Dikin--Stewart--Todd condition measure $\bar\chi(A)$, but \cite{DadushHNV20} observed that the two condition numbers are close to each other. Further, they gave a stronger LLS IPM with running time dependent on the optimal value $\kappa^*(A)$ of $\kappa(A)$ achievable by column rescaling. We refer the reader to the survey \cite{circuitsurvey} for further results related to circuit imbalances and their uses in LP, including also diameter and circuit diameter bounds.

We also note that Fujishige et al.~\cite{fujishige2023update} recently gave a poly$(n,\kappa(A))$ algorithm for the minimum norm point problem \eqref{eq:feasibility} by combining FOMs and active set methods. Their algorithm terminates with an exact solution; on the other hand, it also uses  projection steps that involve solving a system of linear equations. Thus, it is not an FOM; moreover, it is not applicable for optimization LP.

\paragraph{Notation.}
We let $[n]=\{1,2,\ldots,n\}$. For a vector $x\in\R^n$, let $\supp(x)=\{i\in [n]\mid x_i\neq 0\}$ denote its support.
For $A\in \R^{m\times n}$, let $A_i$ denote the $i$-th column of $A$. We use the norm $\|A\|_1=\max_{j\in [n]}\sum_{i=1}^m
|a_{ij}|$,  and the spectral norm $\|A\|_2$.
We normalize the input matrix so that $\|A\|_1\ge 1$.
For a linear subspace $W\subseteq\R^n$, we let $W^\perp\subseteq\R^n$ denote the orthogonal complement.

\paragraph{Overview.}
The remainder of the paper is structured as follows. 
In \Cref{sec:lin_conv}, we discuss preliminaries regarding the convergence guarantees of R-FGM~\cite{necoara2019}. In \Cref{sec:circuits}, we discuss further preliminaries regarding circuit imbalances, proximity and how to  use them algorithmically. In \Cref{sec:main_ideas}, we provide a more detailed overview of our main ideas including formal statements, and then, in \Cref{sec:alg}, we present our algorithm in detailed pseudo-code and prove \Cref{thm::main-result}. \Cref{sec:prox-proof} contains proofs related to the crucial proximity results of \Cref{sec:main_ideas}. We prove the results of \Cref{sec:alg} for the outer and inner routines in \Cref{sec:analysis-outer,sec:analysis-inner}, respectively. Proofs related to certifying the success of our algorithm, which we need for guessing the circuit imbalances, are contained in \Cref{sec:cert-proof}. Finally, in \Cref{sec::examples}, we provide an example showing that a simple self-dual embedding can blow up the Hoffman constant, serving as a motivation for our work.

\section{Linear convergence for functions with quadratic growth}
\label{sec:lin_conv}
Assume we are interested in finding a $\delta$-feasible solution to \ref{eq:LP} or concluding that 
the system is infeasible.
We can use the convex formulation
\begin{equation}
\begin{aligned}
\min \tfrac{1}{2}\|Ax-b\|^2_2\\
\bO\le x\le u\, .
\end{aligned}
\end{equation}
We now outline the running time bounds obtained by
Neocara, Nesterov, and Glineur \cite{necoara2019}.

\begin{definition}
The function $f:\, \R^n\to \R$ is
\emph{$L_f$-smooth} or has \emph{$L_f$-Lipschitz continuous gradient} if $\|\nabla f(x) - \nabla f(y)\|_2 \leq L_f \cdot \|x - y\|_2$ for any $x,y\in\operatorname{dom}(f)$.
\end{definition}
\begin{lemma} \label{lem:Lips}
    The function $f(x) = \frac{1}{2} \|A x - b\|_2^2$ is $\|A\|_2^2$-smooth.
\end{lemma}
\begin{proof}
    This follows as $\|\nabla f(x) - \nabla f(y)\|_2 = \|A^\top A (x - y)\|_2 \leq \|A^\top A\|_2 \cdot\|x - y\|_2=\|A\|_2^2\cdot\|x - y\|_2$.
\end{proof}
\begin{definition}\label{def:quadratic-growth}
Let $f:\R^n\to \R$ be continuously differentiable, let $X\subseteq \operatorname{dom}(f)$ be a closed convex set, and $f^*=\min_{x\in X} f(x)$ the minimum value. Then
    $f$ has \emph{$\mu_f$-quadratic  growth on $X$} if, for any $x\in X$, there exists an optimal solution $\bar{x}$ (that is, $f(\bar{x}) = f^*$) such that $f(x) - f^* \geq \frac{\mu_f}{2} \|x - \bar{x}\|_2^2$.
\end{definition}
\begin{lemma}[{\cite[Theorem 8]{necoara2019}}] \label{lem:qfg-ne}
    The function $f(x) = \frac{1}{2} \|A x - b\|_2^2$ has $1 / \theta^2_{2,2}(A)$-quadratic  growth.
\end{lemma}
The \RFGM{} method proposed in \cite{necoara2019} optimizes a convex function by iteratively applying the standard accelerated projected gradient descent algorithm. \RFGM{} starts with $x^0$ as the initial point and then performs $\Kc$ iterations of accelerated projected gradient descent to obtain $x^{\Kc}$, for a suitable $\Kc$. \RFGM{} then uses $x^{\Kc}$ as the new starting point and repeats the process, performing another $\Kc$ iterations of accelerated projected gradient descent. This process is repeated multiple times.  For a convex function which is $L_f$-smooth and has $\mu_f$-quadratic growth, \cite{necoara2019} shows that, for every $\Kc$ iterations, the difference between the current function value and the optimum is reduced by a factor of $\ee^2$:

\begin{theorem} \label{thm:r-fgm-convergence-rate}
    Suppose function $f$ is $L_f$-smooth and has $\mu_f$-quadratic  growth. Let $\Kc = \lceil 2 \ee \sqrt{L_f / \mu_f} \rceil$ and $x^0$ is the starting point. Then, after $k \cdot \Kc$ iterations, the \RFGM{} method outputs $x$ such that
    \begin{align*}
        f(x) - f^* \leq \ee^{-2 k} (f(x^0) - f^*).
    \end{align*}
\end{theorem}
\section{Circuit imbalances and proximity}

\label{sec:circuits}
For a linear space $W\subset \R^n$, $g\in W$  is an 
 \emph{elementary vector} if $g$ is a support minimal nonzero vector in $W$, that is, no  $h\in W\setminus\{\bO\}$ exists such that $\supp(h)\subsetneq \supp(g)$.
 A \emph{circuit} in $W$ is the support of some elementary vector; these are precisely the circuits in the associated  linear matroid $\mathcal{M}(W)$. We let $\EE(W)\subseteq W$ 
 denote the set of elementary vectors  in the space $W$.

The subspaces  $W=\{\bO\}$ and $W=\R^N$ are called trivial subspaces; all other subspaces are nontrivial.
 We define the \emph{fractional circuit imbalance measure}
 \[
\kappa(W):=\max\left\{\left|\frac{g_j}{g_i}\right|:\, g\in\EE(W), i,j\in \supp(g)\right\}\, 
\]
for nontrivial subspaces, and $\kappa(W):=1$ for trivial subspaces.

Further, if $W$ is a rational linear space, we let $\bar\EE(W)\subseteq \EE(W)$ denote the set of integer elementary vectors $g\in \mathbb{Z}^n\cap \EE(W)$ such that the largest common divisor of the entries is 1. We define the \emph{max circuit imbalance measure} as  
 \[
\bar\kappa(W):=\max\left\{\|g\|_\infty:\, g\in\bar\EE(W)\right\}\, .
\]
When using the term `circuit imbalance measure' without any specification, it will refer to the fractional version. Note that $\kappa(W)\le \bar\kappa(W)$ but they may not be equal. For example, if the single elementary vector up to scaling is $(4,7,8)$, then $\kappa(W)=2$ but $\bar\kappa(W)=8$. 

{Let} $A\in\R^{m\times n}$ be a matrix, and let $W=\ker(A)$ be the kernel space of $A$. We let $\EE(A)$, $\kappa(A)$, $\bar\kappa(A)$  denote  $\EE(W)$, $\kappa(W)$, $\bar\kappa(W)$, respectively, for the kernel space $W=\ker(A)$. We refer the reader to the survey \cite{circuitsurvey} for properties and applications of circuit imbalances. Below, we mention some basic properties.

Recall that a matrix is \emph{totally unimodular (TU)} if the determinant of every square submatrix is $0$, $+1$, or $-1$. We note that  $\kappa(W) = 1$ if and only if there exists a TU matrix $A\in\R^{m\times n}$ such that $W = \ker(A)$. This follows by a 1957 result of Cederbaum. Further, it is easy to verify that for an integer matrix $A\in\mathbb{Z}^{m\times n}$, the inequality $\bar\kappa(A)\le \Delta(A)$ holds, where $\Delta(A)$ is the largest absolute value of a subdeterminant of $A$. However, $\bar\kappa(A)$ can be arbitrarily smaller: 
$\bar\kappa(A)=2$ for the node-edge incidence matrix of any undirected graph, whereas $\Delta(A)$ can be exponentially large. See \cite[Section 3.1]{circuitsurvey} for the above results. We will also use the following important self-duality of $\kappa$:

\begin{lemma}[{\cite{DadushHNV20}}]\label{lem:dual}
Let $W\subseteq \R$ be a linear subspace. Then $\kappa(W)=\kappa(W^\perp)$.
\end{lemma}

\paragraph{Conformal circuit decompositions}
\label{par:sign_consistent}
We say that the vector $y \in \R^n$ \emph{conforms to}
$x\in\R^n$ if $x_i y_i > 0$ whenever $y_i\neq 0$. 
Given a subspace $W\subseteq \R^n$, a \emph{conformal circuit decomposition} of a vector $z\in W$ is a decomposition
$z=\sum_{k=1}^h  g^k$, where
$h\le n$ and  $g^1,g^2,\ldots,g^h\in \EE(W)$ are elementary vectors that conform to
$z$. Further, for each $i=1,2,\ldots,h-1$, $\supp(g^i)\setminus \cup_{j=i+1}^h \supp(g^j)\neq \emptyset$. A fundamental result on elementary vectors asserts the existence of a conformal circuit decomposition, see e.g. \cite{Fulkerson1968,RockafellarTheEV}. Note that there may be multiple conformal circuit decompositions of a vector.

\begin{lemma} \label{lem:conformal}
For every subspace $W\subseteq \R^n$,  every $z\in W$ admits a conformal circuit decomposition. 
\end{lemma}

Given $A\in\R^{m\times n}$, we define the extended subspace $\cX_A\subset \R^{n+m}$ as $\cX_A:=\ker(A\mid -I_m)$. Hence, for every  $z\in \R^n$, $(z,Az)\in\cX_A$. For $z\in \R^n$, a  \emph{generalized path-circuit decomposition of $z$ with respect to $A$} is a decomposition
$z=\sum_{k=1}^h  g^k$,
where $h\le n+m$, and for each $k\in [h]$, $(g^k,Ag^k)\in\R^{n+m}$ is an elementary vector in $\cX_A$ that conforms to $(z,Az)$. Note that this corresponds to
a conformal circuit decomposition of $(z,Az)$ in $\cX_A$.
We say that  $g^k$ is an \emph{inner vector} in the decomposition if $Ag^k=0$ and an \emph{outer vector} otherwise.

We say that $z\in\R^n$ is  \emph{cycle-free with respect to $A$}, if no $y\in\ker(A)$, $y\neq 0$ exists that conforms $z$. Note that this is equivalent to the property that 
all generalized path-circuit decompositions of $z$ contain outer vectors only. The following lemma will play a key role in analyzing our algorithms. 
\begin{lemma}\label{lem:contig-proximity}
Let $A\in\R^{m\times n}$ and let $z\in\R^n$ be
 cycle-free with respect to $A$. Then
\[
 \|z\|_\infty\le \kappa({\cX_A})\cdot\|Az\|_1\, \quad\mbox{and}\quad \|z\|_2\le m\cdot\kappa({\cX_A})\cdot\|Az\|_2\, . 
\]
\end{lemma}
\begin{proof}
Consider a generalized path-circuit decomposition $z=\sum_{k=1}^h g^k$. Since $z$ is cycle-free, for each $g^k$, $Ag^k\neq 0$, and therefore  $|g^k_j|\le \kappa({\cX_A})\cdot|(Ag^k)_i|$ for any $j\in\supp(g^k)$ and $i\in \supp(Ag^k)$.
By the conformity property, $|z_j|=\sum_{k=1}^h |g^k_j|$ for  $j\in [n]$ and $|(Az)_i|=\sum_{k=1}^h|(Ag^k)_i|$ for $i\in [m]$.
Thus, for any $j\in [n]$,
\[
|z_j|=\sum_{k=1}^h |g^k_j|\le  \kappa({\cX_A})\cdot\sum_{i=1}^m \sum_{k=1}^h |(Ag^k)_i|=\kappa({\cX_A})\cdot\sum_{i=1}^m |(Az)_i|=\kappa({\cX_A})\cdot\|Az\|_1\, .
\]
For the second inequality, note that 
$\|g^k\|_2\le \sqrt{m}\cdot \kappa({\cX_A})|(Ag^k)_i|$ for any $k\in [h]$ and $i\in \supp(Ag^k)$, since for any elementary vector $(g^k,Ag^k)\in \cX_A$ with  $\supp(Ag^k)\neq 0$, the columns in $\supp(g^k)$ must be linearly independent, and therefore $|\supp(g^k)|\le m$. This implies 
\[
\|z\|_2\le \sum_{k=1}^h \|g^k\|_2\le \sqrt{m}\cdot \kappa({\cX_A})\cdot\|Az\|_1\le m\cdot \kappa({\cX_A})\cdot\|Az\|_2\, .\hide{\qedhere}
\]\end{proof}
The following lemma is trivial for the input matrix since we assume $ \|A\|_1 \geq 1$. However, we also need this guarantee for its column submatrices in the recursive calls.
\begin{corollary} \label{cor:kappa-A-lower-bound}
    For any non-zero matrix $A\in\R^{m\times n}$, $n \cdot\kappa(\cX_A)\cdot \|A\|_1 \geq 1$.
\end{corollary}
\begin{proof}
Pick a  $z\in\R^n$ such that $Az \neq 0$. Let $\bar z\in\ker(A)$ with $\|z-\bar z\|_2$ minimal. Then $z - \bar{z}$ is cycle-free and $z - \bar{z} \neq 0$. The result follows as $\|z - \bar{z}\|_\infty\le \kappa({\cX_A})\cdot\|A(z-\zbar)\|_1  \le n\cdot \kappa(\cX_A)\cdot \|A\|_1 \cdot\|z - \bar{z}\|_\infty$.
\end{proof}

When $A$ is clear from the context, we simply use $\kappa=\kappa(\cX_A)$. If $A$ is a node-arc incidence matrix of a directed graph, then $A$, and consequently also $(A\mid-I_m)$ is a TU matrix, implying $\bar\kappa(\cX_A)=1$. For undirected graph incidence matrices, one can show $\bar\kappa(\cX_A)\le 2$.

\begin{lemma}\label{lem:hoffman-kappa} Let  $A\in\R^{n\times m}$. Then $\theta_{1,\infty}(A)\le \kappa(\cX_A)$ and  $\theta_{2,2}(A)\le m\cdot\kappa(\cX_A)$.
\end{lemma}
\begin{proof}
    We want to show  for any $x\in [\bO,u]$, and any $b\in\R^m$, whenever $\cP_{A,b,u}$ is nonempty, there exists an $\bar x\in\cP_{A,b,u}$ such that $\|\bar x-x\|_\infty\le \kappa(\cX_A) \|Ax-b\|_1$ and 
    $\|\bar x-x\|_2\le m\cdot\kappa(\cX_A) \|Ax-b\|_2$. We select $\bar x$ as the nearest feasible point to $x$ in $\ell_2$-norm. Thus, $A\bar x= b$. We claim that $\bar x-x$ is cycle-free with respect to $A$. Indeed, if a generalized path-circuit decomposition of $\bar x-x$ contained an inner vector $g^k$, then $\bar x'=x+g^k$ would also be feasible, with $\|\bar x'-x\|_2<\|\bar x-x\|_2$. Thus, Lemma~\ref{lem:contig-proximity} can be applied with $z=x - \bar x$ and the claim follows.
\end{proof}
The first inequality may be tight. Assume there exists an elementary vector $(g,Ag)$ in $\cX_A$ such that $|g_i|=\kappa(\cX_A)$ for all $i\in \supp(g)$, $|\supp(Ag)|=1$, and $(Ag)_i=1$ for the nonzero component. Further, let $b=0$, and let $u_i=0$ for all $i\notin \supp(g)$. Since $(g,Ag)$ is a support minimal nonzero vector in $\cX_A$, it follows that the only feasible solution to $A\bar x=0$, $\bar x\in [\bO,u]$ is $\bar x=0$. Thus, we get a tight example with $\theta_{1,\infty}(A)= \kappa(\cX_A)$. The same example shows tightness of the second inequality up to a factor $\sqrt{m}$. 

Our algorithm will also find dual certificates. The next lemma shows that the corresponding dual systems have the same circuit imbalances.
\begin{lemma}\label{lem:cX-orth}
For any matrix $A\in\R^{m\times n}$, $\kappa(A^\top | I_n)=\kappa(\cX_A)$.
\end{lemma}
\begin{proof}
Recall that $\cX_A=\ker(A|-I_m)$. It is easy to verify that $\ker(I_n | A^\top)$ is the orthogonal complement of $\cX_A$. The statement then follows by Lemma~\ref{lem:dual}, noting also that reordering the columns does not change the circuit imbalances.
\end{proof}

\subsection{Guessing the circuit imbalances}
\label{sec:circuit-hard}

Given a matrix $A\in\R^{m\times n}$, there is no hope of 
getting any reasonable approximation of the circuit imbalance values. It is NP-hard to approximate $\kappa(A)$ within a factor $2^{\mathrm{poly}(m)}$ for $A\in \R^{m\times n}$, see \cite{DadushHNV20}, using a result of Tun{\c{c}}el \cite{Tuncel1999} on the related condition number $\bar\chi(A)$. However, our algorithms make use of the values $\kappa(A)$ and $\bar\kappa(A)$. 

Nevertheless, one can use a guessing procedure that guarantees the same asymptotic running times without knowing these values. 
 First, note that upper bounds rather than exact values suffice throughout.
Algorithm~\ref{alg:outermost} below makes recursive calls to column submatrices of the original input matrix $A$. 
 If $\bar\kappa(A)\le \hat\kappa$ for the input matrix $A$, then $\hat\kappa$ is an upper bound on all the circuit imbalance values seen in the recursive calls.

As our initial estimate, we set $\hat\kappa=1$.
When run with a correct guess $\hat\kappa \ge \bar\kappa(A)$, 
Algorithm~\ref{alg:outermost} returns approximately optimal primal and dual solutions to \ref{eq:LP}. Running it with an incorrect guess may lead to 
a failure while running the algorithm: either R-FGM does not find a solution of the required accuracy within the given number of steps, or the final primal and dual solutions do not satisfy the required approximate feasibility and complementarity properties. 
We can easily detect both kind of failures.
If no failure is detected, then the primal and dual solutions certify approximate optimality for each other; this may  happen even when $\hat\kappa<\bar\kappa(A)$. Each time we detect a failure, we double the estimate $\hat\kappa$ and restart the algorithm. 

The overall running time  bound in Theorem~\ref{thm::main-result} is also the total running time of this process, because at each call we have $\hat\kappa\le 2\bar\kappa(A)$, and the running time bound of the final run dominates the running time bound of all previous runs. 

\section{Main ideas and key statements}
\label{sec:main_ideas}

Before describing the algorithm in Section~\ref{sec:alg}, we highlight the key ideas and formulate the main underlying proximity results. We gradually reduce 
 \ref{eq:LP} by fixing some variables to their upper or lower bounds, and replacing the cost vector by an equivalent one of smaller norm.  We first start by describing the simpler feasibility algorithm. The optimization algorithm has two components: the outer loop and the inner loop.

\subsection{The feasibility algorithm}
We first show how the \RFGM{} algorithm from \cite{necoara2019} leads to a simple  algorithm for finding a $\delta$-feasible solution. Here, we assume that the LP is feasible. In the proof of Theorem~\ref{thm::main-result} in Section~\ref{sec:main-proof}, we explain how this assumption can be removed in general.

\begin{theorem}\label{thm:feas}
There is an algorithm \Feasiblealg$(A,b,u,\delta)$, which, on input $A\in\R^{m\times n}$, $\|A\|_1\ge 1$, $b\in\R^m$, $u\in\R^n$, supposing the system $Ax=b$, $x\in [\bO,u]$ is feasible, finds a $\delta$-feasible solution using  $O\big(m \|A\|_2 \cdot\kappa(\cX_A)\cdot \log(m \|b\|_1 / \delta)\big)$ iterations of \RFGM{}.
\end{theorem}
\begin{proof}
Let $f(x)=\frac{1}{2}\|Ax-b\|^2$. 
We  use the \RFGM{} algorithm from \cite{necoara2019} to find an $\varepsilon$-approximate minimizer of $f(x)$ over $x\in [\bO,u]$, i.e., 
the system \eqref{eq:feasibility}, where $\varepsilon:=\|A\|_1^2 \cdot \delta^2 / (2m)$. We choose $x=0$ as the starting point. By Theorem~\ref{thm:r-fgm-convergence-rate},  $O(\|A\|_2 \cdot \theta_{2,2}(A)\cdot \log(m \|b\|_1/ \delta))$ iterations suffice (using $\|A\|_1\ge 1$). The  bound on the number of iterations follows, for by Lemma~\ref{lem:hoffman-kappa},  $\theta_{2,2}(A) \leq m \cdot\kappa(\cX_A)$. 

By the assumption that the system is feasible, the optimum value is $0$. Thus, an $\varepsilon$-approximate solution has $f(x)\le \varepsilon$, which yields a $\delta$-feasible solution. 
\end{proof}

\subsection{The outer loop}

In the outer loop, our  goal is to find a $\delta^\feas$-feasible and $\delta^\opt$-optimal solution to \ref{eq:LP}. We distinguish these two accuracy parameters for the sake of the recursive algorithm, where the required feasibility and optimality accuracies need to be changed differently in the recursive calls.

\paragraph{Primal-dual optimality and cost shifting.}
We use primal-dual arguments, making variable fixing decisions based on approximate complementarity conditions. 
The dual to \ref{eq:LP} can be written as 
\begin{equation}\label{eq:LP-dual}\tag{{Dual$(A,b,c,u)$}}
\begin{aligned}
\max\ \pr{b}{\pi}-&\pr{u}{w^+}\\
A^\top \pi+w^--w^+&= c\,\\
w^-,w^+&\ge 0\, .
\end{aligned}
\end{equation}
Note that given $\pi\in\R^m$, the unique best choice of the variables $w^-$ and $w^+$ is $w^-=\max\{c-A^\top \pi,0\}$ and $w^+=\max\{A^\top \pi-c,0\}$.
When we speak of a dual solution $\pi\in\R^m$, we mean its extension with these variables.
Recall the primal-dual optimality conditions: $x^*\in \cP_{A,b,u}$  and $\pi\in\R^m$  are optimal respectively to \ref{eq:LP} and to \ref{eq:LP-dual} if and only if the following holds: 
\begin{equation}\label{eq:LP-compl}
\mbox{if $A_i^\top \pi<c_i$ then $x_i=0$, and if $A_i^\top \pi>c_i$ then $x_i=u_i$ for every $i\in[n]$.}
\end{equation}
Also note that we can naturally shift the cost function for any $\pi\in\R^m$ as stated in the next lemma. 
\begin{lemma}
\label{lem::shift-cost-fn-in-LP}
\ref{eq:LP} has exactly the same solutions and the same optimum value as the following linear program: 
\begin{equation}\label{eq:shifted}
\begin{aligned}
\min &\pr{c - A^\top \pi}{x} + \pr{b}{\pi}\, \quad \mbox{s.t.}\quad Ax = b\, ,\quad x\in[\bO,u]\, .
\end{aligned}
\end{equation}
\end{lemma}

\paragraph{Approximate complementarity and proximity.}
Assume now that we have a pair of primal and dual solutions $x$ and $\pi$ that do not satisfy complementarity, but we have a quantitative bound on the violation. Namely, for a suitably chosen threshold $\sigma\ge 0$, let 
\[
\begin{aligned}
\theta(x,\pi,\sigma)&:=\sum_{c_{i} -A_i^\top  \pi> \sigma}x_{i} +\sum_{c_{i} -A_i^\top  \pi< -\sigma} (u_{i}-x_{i})\, ,\quad \mbox{and} \\
\bJ(\pi,\sigma)&:=\left\{i\in [n]:\,  |c_{i} -  A_i^\top \pi|>n\cdot\lceil\kappa(\cX_A)\rceil\cdot\sigma\right\}\,
 .
\end{aligned}
\]
Note that if $x$ and $\pi$ are primal and dual optimal, then the primal-dual complementarity constraints~\eqref{eq:LP-compl} imply $\theta(x,\pi,0)=0$. Let us assume that for some $\sigma>0$, this quantity is still small. Note also that $\bJ(\pi,\sigma)$ is the set of indices where the absolute value of the slack is much higher than the threshold $\sigma$. In particular, $\min\{x_i,u_i-x_i\}\le \theta(x,\pi,\sigma)$ on these indices. Our key proximity result asserts that there exists an optimal solution that is close to the current solution on these indices. The proof is deferred to Section~\ref{sec:prox-proof}.
 
\begin{lemma}\label{lem:close_opt_sol-kappa}
Let $x\in \cP_{A,b,u}$ be a feasible solution. Then\ there exists an optimal solution $x^*$ for \ref{eq:LP} such that
\[
|x_i-x^*_i|\le \kappa(\cX_A) \cdot \theta(x,\pi,\sigma) 
\]
for all $i\in \bJ(\pi,\sigma)$.
\end{lemma}
\paragraph{Variable fixing.} Assume that from the inner loop of the algorithm we get $x\in[\bO,u]$ and $\pi\in\R^m$ such that the feasibility violation $\|Ax-b\|_1$ and the complementarity violation $\theta(x,\pi,\sigma)$ are both tiny for the choice $\sigma := \|c\|_{\infty} / (4 n  \lceil\kappa(\cX_A)\rceil)$. Note that the for this choice, the threshold in the definition of $\bJ(\pi,\sigma)$ becomes $\|c\|_\infty/4$. We partition $\bJ(\pi,\sigma)$ into
\[
\begin{aligned}
J_1:=\left\{i\in \bJ(\pi,\sigma) ~\big|~ c_i-A_i^\top \pi<-\tfrac{\|c\|_\infty}4\right\}\, ,\quad\mbox{and}\quad J_2&:=\left\{i\in \bJ(\pi,\sigma)~\big|~ c_i-A_i^\top \pi>\tfrac{\|c\|_\infty}4\,\right\}\, ,
\end{aligned}
\]
We apply Lemma~\ref{lem:close_opt_sol-kappa} to the problem with the modified right hand side $b'=Ax$. By ensuring that $\theta(x,\pi,\sigma)$ is sufficiently small, we will see that there is an optimal solution  $x^*$ with $x^*_i\approx 0$ for $i\in J_1$ and $x_i^*\approx u_i$ for $i\in J_2$. 

We fix these variables to the lower and upper bounds, respectively, and shift the cost function according to $\pi$. Thus, we specify the following new LP.
Let  $N:=[n]\setminus (J_1\cup J_2)$ and $\bar b:=\AN\xN$. 
\begin{equation}\label{eq:lp-cnew}\tag{LP$(\AN,\bar b,\cN - \AN^\top \pi,\uN)$}
\begin{aligned}
\min\ & \pr{\cN - \AN^\top \pi}{z}\\
&~~\AN z=\bar b\\
&\bO_{\scriptscriptstyle{\!N}}\le z\le \uN \, .
\end{aligned}
\end{equation}
We show the following result, which says the optimal solution of \ref{eq:lp-cnew} provides an approximately feasible and optimal solution to \ref{eq:LP}. The approximation is in terms of $\theta(x,\pi,\sigma)$ and  $\|Ax-b\|_1$. Recall that $\optv(A,b,c,u)$ denotes the optimal value, the value achieved by the solution to \ref{eq:LP}. The proof is given in Section~\ref{sec:prox-proof}.
\begin{theorem}\label{thm:variable-fixing}
For $A\in\R^{m\times n}$, $b\in\R^m$, and $c,u\in\R^n$ such that \ref{eq:LP} is feasible, 
let $\sigma:=\|c\|_\infty/(4n\cdot\lceil\kappa(\cX_A)\rceil)$, and let $x\in[\bO,u]$ and $\pi\in\R^m$ be a pair of (not necessarily feasible) primal and dual solutions. Then, \ref{eq:lp-cnew}
is feasible and, in addition satisfies the following:
\begin{itemize}
    \item  feasibility condition: \begin{align}\label{eqn::feasibility-bound}
\|b-\bar b-\AJtwo \uJtwo\|_1\le \theta(x,\pi,\sigma)\cdot\|A\|_1+\|Ax-b\|_1\, ,
\end{align}
\item   optimality condition:
\begin{align}
\label{eqn::optimality-bound}
\nonumber
&\big|\optv(\AN,\bar b,\cN - \AN^\top\pi,\uN) +  \pr{\bar b}{\pi}+\pr{\cJtwo}{\uJtwo}-\optv(A,b,c,u)\big|
\nonumber\\
 &= \big|\optv(\AN,\bar b,\cN,\uN)+\pr{\cJtwo}{\uJtwo}-\optv(A,b,c,u)\big|\nonumber\\
&~~~~~~\le \kappa(\cX_A)\cdot\|c\|_1\cdot \|Ax-b\|_1 + |J_1\cup J_2|\cdot\kappa(\cX_A)\cdot\|c\|_{1}\cdot
    \big(2 + \kappa(\cX_A) \|A\|_1\big)\cdot \theta(x,\pi,\sigma)\, ,
\end{align}
\item cost reduction: $\|c_N-A_N^\top \pi\|_\infty\le \|c\|_\infty/4$.
\end{itemize}
\end{theorem}

With this theorem, if one can find a pair $(x, \pi)$ such that the right hand sides of inequalities \eqref{eqn::feasibility-bound} and \eqref{eqn::optimality-bound} are tiny, then \ref{eq:LP} can be reduced to \ref{eq:lp-cnew} with a tiny loss on feasibility and optimality. Moreover, each iteration reduces the $\ell_\infty$-cost on the remaining variables by a factor $4$.
One can repeat this procedure and ultimately reduce the original problem to one with an extremely small objective function value and possibly with fewer variables. Solving this problem will give a good enough solution to the original \ref{eq:LP}, after restoring any variables fixed to the lower or upper bounds. 

It is possible that both $J_1$ and $J_2$ are empty. This means that $\|c-A^\top \pi\|_\infty\le \|c\|_\infty / 4$; we can simply recurse with the same $b$ but improved cost function. Note that we could make progress more agressively by a preprocessing step that projects $c$ to the kernel of $A$; this gets a cost vector of the form $c'=c-A^\top\pi$ with the smallest possible $\ell_2$-norm---such a preprocessing is used in the strongly polynomial algorithms \cite{Tardos85,Tardos86,DadushNV20}. Setting a slightly smaller $\sigma$ would then guarantee variable fixing in every iteration. However, the projection amounts to solving a system of linear equations that may be computationally more expensive. We instead proceed with lazier updates as above.

\subsection{The inner loop} 

Next, we  describe our approach for obtaining a pair $(x, \pi)$ such that the right hand sides of inequalities \eqref{eqn::feasibility-bound} and \eqref{eqn::optimality-bound} are tiny, which is the purpose of the inner loop. 
For this, we need to guarantee that 
$\theta(x,\pi,\sigma)$ and  $\|Ax-b\|_1$ are sufficiently small. We use a potential function $F_{\tau}(x)$ of the form \eqref{eq:F-tau-intro} for a modified cost function $\ceps$. 

As noted in the introduction, if $\ceps=c/\|c\|_\infty$, and  $\tau$ is within $\delta/2$ of the optimum value of \ref{eq:LP}, then a $\delta'$-approximate minimizer to $F_{\tau}(x)$ for a suitably chosen $\delta'$ would immediately give a $\delta$-approximate and $\delta$-feasible solution to \ref{eq:LP}. Thus, we would not need the outer loops; a binary search on $\tau$ and using the feasibility algorithm on this system would already give the desired solution, without the need for variable fixings in the outer loop. 

However, the Hoffman-constant corresponding to the function \eqref{eq:F-tau-intro} with $\ceps=c/\|c\|_\infty$ could be unbounded in terms of $\bar\kappa(\cX_A)$ if $c$ can be arbitrary, as discussed in Section~\ref{sec::examples}.
To circumvent this problem, we discretize $c/\|c\|_\infty$ into integer multiples of $\varepsilon=1/(8n\cdot\lceil\kappa(\cX_A)\rceil)=\sigma/(2\|c\|_\infty)$.

Using the discretized $\ceps$, for a suitable choice of $\tau$, we can guarantee \eqref{eqn::feasibility-bound} and \eqref{eqn::optimality-bound}, that is, bound $\|Ax-b\|_1$ and $\theta(x,\pi,\sigma)$, where the dual solution $\pi$ is defined based on the gradient of $\Ft{\tau}(x)$ as $\pi:=\frac{\|c\|_{\infty}}{\|A\|^2_1\alpha}(b-Ax)$ for $\alpha:=\max\{0, \langle \ceps, x \rangle - \tau\}$.

To bound the infeasibility $\|Ax-b\|_1$, we need to find a solution $x$ where $\Ft{\tau}(x)$  is small, since $\|Ax-b\|_1\le (2n\|A\|_1 \Ft{\tau}(x))^{1/2}$. Therefore, $\tau$ should not be much smaller than the optimum value $\optv(A,b,c,u)$ of \ref{eq:LP}. 
The bound on $\theta(x,\pi,\sigma)$ can be shown by arguing that the improving directions of the gradient are small at an approximately optimal solution $x$:  $x_i \approx 0$ if $\nabla_i \Ft{\tau}(x) \gg 0$ and $x_i \approx u_i$ if $\nabla_i \Ft{\tau}(x) \ll 0$, and that $|c_i-\ceps_i| \cdot \|c\|_\infty \le \|c\|_\infty\cdot \varepsilon=\sigma/2$. We also need that $\alpha>0$ and is not too small. Based on these requirements, we can establish a narrow (but not too narrow) interval of $\tau$ where a sufficiently accurate approximate solution to $\Ft{\tau}(x)$ exists. Using that $\Ftopt{\tau} := \min\{\Ft{\tau}(x)\mid x\in[\bO,u]\}$ is a Lipschitz-continuous and non-increasing continuous function in~$\tau$, we can find a suitable $\tau$ by binary search.

\subsection{Dual certificates}\label{sec:dual-cert}

As discussed above, our goal is not just to find a  $\delta^\feas$-feasible and $\delta^\opt$-optimal solution, but also a dual certificate for the latter property. This is important as it enables us to verify the correctness of the solution, which is only guaranteed when using  an estimate $\hat\kappa\ge\bar\kappa$.

\begin{definition}\label{def:dual-cert}
Let $A\in\R^{m\times n}$, $b\in\R^m$,
$c,u\in\R^n$ and  $\delta\ge 0$, 
and let  
$x\in [\bO,u]$. We say that $(\pi,w^+,w^-)\in\R^{m\times n\times n}$ is a \emph{$\delta$-certificate for $x$}, if 
\begin{enumerate}[(i)]
    \item $A^\top \pi+w^--w^+=c$, 
    \item $0\le w_i^-\le 2\delta\|c\|_\infty/x_i$, $0\le w_i^+\le 2\delta\|c\|_\infty/(u_i-x_i)$ for all $i\in [n]$, and
    \item $\|\pi\|_\infty \le 2\delta\|c\|_\infty/\|Ax-b\|_1$.
\end{enumerate}
\end{definition}
Our next lemma shows that $\delta$-certificates indeed certify approximate optimality, and conversely, for every approximately optimal solution, such a certificate can be found. The proof is given in Section~\ref{sec:cert-proof}.
\begin{lemma}\label{lem:delta-cert}
Let $A\in\R^{m\times n}$, $b\in\R^m$,
$c,u\in\R^n$, and let $x\in [\bO,u]$.
\begin{enumerate}[(i)]
\item If there is a $\delta$-certificate for $x$ for some $\delta\ge 0$, then $x$ is $(4n+2)\delta$-optimal.
\item Suppose  $0\le \delta^\feas \cdot n\cdot\kappa(\cX_A) \cdot \|A\|_1 \le \delta^\opt$. If $x$ is a $\delta^\feas$-feasible and $\delta^\opt$-optimal solution, then there exists a $\delta^\opt$-certificate for $x$.  
\end{enumerate}
\end{lemma}
Our next theorem justifies this concept and provides certificates efficiently. The proof is also deferred to 
Section~\ref{sec:cert-proof}.

\begin{theorem}\label{thm:dual-check}
Suppose $A\in\R^{m\times n}$, $\|A\|_1\ge 1$, $b\in\R^m$, $x,c,u\in\R^n$, $ 0 \leq \delta^\feas \cdot n\cdot\kappa(\cX_A) \cdot \|A\|_1 \le \delta^\opt$, and $x\in[\bO,u]$ is both $\delta^\feas$-feasible and $\delta^\opt$-optimal.
Then there is an algorithm \Checkalg$(x,A,b,c,u,\delta^\feas,\delta^\opt)$ which on such inputs finds a 
$2\cdot\delta^\opt$-certificate for $x$ in $O\big(m \|A\|_2 \cdot\kappa(\cX_A)\cdot \log(n \|c\|_1 / \delta^\opt)\big)$ iterations of \RFGM{}.
\end{theorem}

\section{The algorithm}\label{sec:alg}

We describe our algorithm under the simplifying assumption that the exact values of $\kappa(A)$ and $\bar\kappa(A)$ are known for the input matrix as well as all submatrices obtained by column deletions. As discussed in 
Section~\ref{sec:circuit-hard}, even though these quantities cannot be computed, one can obtain the same asymptotic running time bounds by repeatedly guessing an estimate $\hat\kappa$.

\subsection{The outer loop: variable fixing}\label{sec:outer}

\Cref{alg:outermost} takes as input $(A,b,c,u)$ such that \ref{eq:LP} is feasible, and accuracy parameters $\delta^\feas$ and $\delta^\opt$ such that $0\le\delta^\feas \cdot 8 n \sqrt{m} \cdot\kappa(\cX_A)\cdot \|A\|_1 \leq  \delta^\opt$. 

Our overall goal in Theorem~\ref{thm::main-result} is to find a $\delta$-feasible and $\delta$-optimal solution to \ref{eq:LP}, without the feasibility assumption. We show in Section~\ref{sec:main-proof} how this can be derived by choosing the accuracy parameters suitably.
\Cref{alg:outermost}  uses a subroutine $\getPair(A,b,c,u,\delta^\feas,\delta^\opt)$ specified as follows.

\begin{center} \fbox{\begin{minipage}{0.7\textwidth} \noindent
{\sf Subroutine} {\getPair} \\
{\bf Input:}  $A\in \R^{m\times n}$, $b\in\R^m$, $c,u\in\R^n$, $\delta^\feas,\delta^\opt>0$ such that $(A,b,c,u)$ is feasible, $\delta^\opt \le \|u\|_1$, and $\delta^\feas \|A\|_1 \leq \delta^\opt$.\\
{\bf Output:}  
 $(x,\pi)$, $x\in [\bO,u]$, $\pi\in\R^m$ such that
\begin{itemize}
    \item The right-hand side of the feasibility bound \eqref{eqn::feasibility-bound} in \Cref{thm:variable-fixing} is at most~$\delta^\feas\|A\|_1/n$.
    \item The right-hand side of the optimality bound \eqref{eqn::optimality-bound} in \Cref{thm:variable-fixing} is at most $\delta^\opt\|c\|_\infty/n$.
    \item $\|\pi\|_\infty\le 4 n \sqrt{m}\cdot \kappa(\cX_A) \cdot\|c\|_{\infty}$.
\end{itemize}
\end{minipage}}\end{center}

\begin{algorithm}
    \caption{$\solveLP$\label{alg:outermost}}
\Input{$A\in\R^{m\times n}$, $b\in\R^m$, $c,u\in\R^n$, 
$0<\delta^\feas (8 n \sqrt{m} \cdot\kappa(\cX_A) \|A\|_1 )\le\delta^\opt$.}
\Output{A $\delta^\feas$-feasible and $\delta^\opt$-optimal solution to \ref{eq:LP} along with a $2\delta^\opt$-certificate} 
    \If{$\delta^\opt\ge\|u\|_1$}{$\bar x\gets\Feasiblealg(A,b,c,u,\delta^\feas)$~~\text{{(See \Cref{thm:feas}.)}} \;
    \Return{$\bar x$}}
    \Else{
    $(x,\pi)\gets \getPair(A,b,c,u,\delta^\feas,\delta^\opt)$ \;
    Define $J_1$, $J_2$, $N$, $\bar b$ as for \Cref{thm:variable-fixing} \;
    $c^\new\gets c - A^\top \pi$ \;
    $\gain\gets\frac{\|\cN\|_\infty}{2\|\cN^\new\|_\infty}$ \; 
    \If{$J_1\cup J_2 = \emptyset$}{
        $x^\outindex\gets\solveLP(A,b,c^\new,u,\delta^\feas,\gain\delta^\opt)$\label{line:nofixing}
    }
    \Else{
        $\xJone^\outindex\gets0$\ ; $\xJtwo^\outindex\gets\uJtwo$ \;
        $\xN^\outindex\gets\solveLP(\AN,\bar b, \cN^\new, \uN, \delta^\feas\cdot |N|/n,\gain\delta^\opt\cdot |N|/n)$ \;
        $\bar\pi\gets\Checkalg(x,A,b,c,u,\delta^\feas,\delta^\opt)$~~\text{(See \Cref{thm:dual-check}.)} \;
    }
    \Return{$(x^\mathrm{out},\bar\pi)$}}
\end{algorithm}

The following theorem provides the analysis of the outer loop. The proof is deferred to \Cref{sec:analysis-outer}.

\begin{theorem}\label{thm:outer-main}
    If \ref{eq:LP} is feasible, then \Cref{alg:outermost} returns a $\delta^\feas$-feasible solution that is $\delta^\opt$-optimal along with a $2 \delta^\opt$-certificate. It {makes}  at most $\log_2(n\|u\|_1/\delta^\opt)$ many recursive calls.
\end{theorem}

\subsection{The inner loop: fast gradient with binary search}\label{sec:inner}

We now describe $\getPair(A,b,c,u,\delta^\feas,\delta^\opt)$, introduced at the beginning of Section~\ref{sec:outer}.
The subroutine needs to output primal and dual vectors $(x,\pi)$ satisfying the feasibility and optimality bounds in \Cref{thm:variable-fixing}. In particular, we need to bound $\theta(x, \pi, \sigma)  = \sum_{i : c_i + A_i^\top \pi > \sigma} x_i + \sum_{i: c_i + A_i^\top \pi < - \sigma} (u_i - x_i)$ and $\|A x - b\|_1$
for $\sigma=\|c\|_\infty/(4n\cdot\lceil\kappa(\cX_A)\rceil)$.
Let us define the accuracy parameter
\begin{equation}\label{eq:epsilon-def}
\varepsilon := \frac{1}{8n\cdot\lceil \kappa(\cX_A)\rceil}=\frac{\sigma}{2\|c\|_\infty}\, .
\end{equation}
Thus, $1/\varepsilon$ is integer.
For some parameter $\tau\in\R$, we use the potential function
\begin{equation}\label{eq:F-def}
\Ft{\tau}(x) := \frac{1}{2}(\max\{0, \langle \ceps, x \rangle  - \tau\})^2 + \frac{1}{2\|A\|_1^2} \|Ax - b\|_2^2\, ,
\end{equation}
where the rounded cost function $\ceps$ is defined by taking the normalized vector $c/\|c\|_\infty$, and rounding each entry to the nearest integer multiple of $\varepsilon$ in the direction of the 0 value
(i.e.,\ rounding down the positive entries and rounding up the negative entries). Recalling that $1/\varepsilon$ is an integer, we have $\|\ceps\|_\infty= 1$.

Let $\Ftopt{\tau} := \min\{\Ft{\tau}(x)\mid x\in[\bO,u]\}$ denote the optimum value. We say that $x\in[\bO,u]$ is a \emph{$\zeta$-approximate minimizer} of $\Ft{\tau}$ if $F(x)\le\Ftopt{\tau} + \zeta$. The following proposition is immediate.

\begin{proposition}\label{prop:Fct-basic}
$\Ftopt{\tau}$ is a non-increasing continuous function of $\tau$. If LP$(A,b,\ceps,u)$ is feasible, then $\optv(A, b, \ceps, u)$ is the smallest value of $\tau$ such that $\Ftopt{\tau} = 0$. 
\end{proposition}

The main driver of our algorithm is Necoara, Nesterov, and Glineur's R-FGM algorithm, applied to $\Ft{\tau}$. We specify this subroutine as follows.
\begin{center} \fbox{\begin{minipage}{0.7\textwidth} \noindent
{\sf Subroutine} {\RFGM} \\
{\bf Input:}  $A\in \R^{m\times n}$, $b\in\R^m$, $c,u\in\R^n$, $\tau\in\R$, $\zeta>0$.\\
{\bf Output:}  
 A $\zeta$-approximate minimizer $x\in [\bO,u]$ of $\Ft{\tau}$.
\end{minipage}}\end{center}

The purpose of $\getPair{}$ is to identify a value $\tau$ by binary search that is slightly below $\optv(A,b,\ceps,u)$. We show that there is a suitable $\tau$ such that a sufficiently accurate approximate minimizer of $\Ft{\tau}$ returns the required primal solution $x$. Moreover, we can also construct the dual $\pi$ from the gradient of $\Ft{\tau}$ at this point.

We define some further parameters to calibrate the accuracy used in the algorithm.
\begin{equation}\label{eq:def-para}
 \paraK :=  n\sqrt{m}\cdot\kappa({\cX_A}) \cdot\|A\|_1\, ,\quad  \paraBarK  :=  64n \paraK\cdot \kappa(\cX_A)=64 n^2 \sqrt{m} \cdot\kappa^2(\cX_A) \cdot \|A\|_1\, ,\quad  
 \targetaccuracy := {\Big(\frac{\deltast}{4 \kappa^2(\cX_A)n^4 \paraBarK \sqrt{m}  }\Big)^2}\, .
\end{equation}
These admit the following simple lower bounds.

\begin{lemma}\label{lem:K-lower}
$\paraK \ge \sqrt{m}$, $\paraBarK \ge 64\sqrt{m}$, and $\paraBarK\sqrt{\targetaccuracy} = {\frac{\deltast} {4n^4\sqrt m\cdot\kappa^2(\cX_A)}}$.
\end{lemma}
\begin{proof}
The first bound follows by
   Corollary~\ref{cor:kappa-A-lower-bound}, and the
  second bound from the first and using $\kappa(\cX_A)\ge 1$. The third bound is immediate from the definition.
  \end{proof}

Algorithm $\getPair$ is shown in Algorithm \ref{alg:getPair}.
\begin{algorithm}
\SetKwFor{Loop}{Repeat}{}{}
    \caption{$\getPair$}\label{alg:getPair}
    \Input{$A\in \R^{m\times n}$, $b\in\R^m$, $c,u\in\R^n$, and $\delta^\feas,\delta^\opt>0$ such that $(A,b,c,u)$ is feasible, and $\delta^\feas \|A\|_1 \leq \delta^\opt< \|u\|_1$}.
    \Output{$x\in [\bO,u]$ and $\pi\in\R^n$.}
$\tau^{+} \la \|u\|_1$ and $\tau^{-} \la - \|u\|_1 - 2 \paraBarK \sqrt{ \targetaccuracy}$ \;
    \Loop{}{
  $\tau \la \frac{\tau^{+} + \tau^{-}}{2}$ \;
    $x\la \RFGM(A,b,c,u,\tau,\targetaccuracy)$ \;
     \lIf{$\Ft{\tau}(x) > 2 \paraBarK ^2 \targetaccuracy$ }{  $\tau^{-} \la \tau$}
    \lIf{$\Ft{\tau}(x) <  \paraBarK ^2 \targetaccuracy$ }{  $\tau^{+} \la \tau$}
     \If{$\Ft{\tau}(x) \in [\paraBarK ^2 \targetaccuracy, 2 \paraBarK ^2 \targetaccuracy]$ }{ $\alpha\la \max\{0, \langle \ceps, x \rangle - \tau\}$ \;
    $\pi\la \frac{\|c\|_{\infty}}{\|A\|^2_1\alpha}(b-Ax)$ \;
     $w^+ \la \max\{A^\top \pi - c,0\}$ ;
     $w^- \la \max\{c - A^\top \pi,0\}$ \;
     \Return{$(x, \pi)$}
}
 }   
\end{algorithm}

\begin{theorem}\label{thm:getpair-main}
Assume \ref{eq:LP} is feasible.
Algorithm~\ref{alg:getPair} makes
$O\big(\log\big[\|u\|_1 nm\cdot\kappa(\cX_{A})/\deltast\big]\big)$ calls to \RFGM{}, and altogether these calls use
$O\big( n^{1.5} m^2  \|A\|^2_1 \cdot\bar\kappa^3(\cX_{A}) \log^2 \big[\|u\|_1 n m \cdot \\|A\|_1 \kappa(\cX_A)/\deltast\big]\big)$ iterations. 
On terminating, it outputs $(x, \pi)$ satisfying:
\begin{enumerate}[{(i)}]
    \item  $\theta(x, \pi, \sigma) \cdot \|A\|_1 + \|Ax - b\|_1 \leq \delta^\feas \|A\|_1 / n$.
    \item $\kappa(\cX_A)\cdot\|c\|_1\cdot \|Ax-b\|_1 + |J_1\cup J_2|\cdot\kappa(\cX_A)\cdot\|c\|_{1}\cdot
    \big(2 + \kappa(\cX_A) \|A\|_1\big)\cdot \theta(x,\pi,\sigma)  \leq \delta^\opt \|c\|_\infty / n $.
    \item  $\|\pi\|_{\infty} \leq 4n \sqrt{m} \cdot \kappa(\cX_A)\cdot  \|c\|_{\infty} $.
\end{enumerate}
\end{theorem}

\subsection{Putting everything together}\label{sec:main-proof}

We now combine the above ingredients to prove Theorem~\ref{thm::main-result}.

\begin{proof}[Proof of Theorem~\ref{thm::main-result}]
First, let us assume that \ref{eq:LP} is feasible.
We set $\delta^\opt = \delta / (8 n + 4)$ and start with the guess $\hat \kappa = 1$.  Then, we set $\delta^\feas = \delta^\opt / (8 n \sqrt{m} \hat \kappa \|A\|_1)$, and run the algorithm $\solveLP(A, b, c, u, \delta^\feas, \delta^\opt)$. If it succeeds and outputs a primal solution $x$ and a certificate $\pi$, then we can check if the solution is a $\delta^\feas$-feasible solution by checking the constraint $\|Ax - b\|_1 \leq \delta^\feas \cdot \|A\|_1$; and check  if $\pi$ is a $2 \delta^\opt$-certificate by checking the constraints in Definition~\ref{def:dual-cert}. If $x$ is a $\delta^\feas$-feasible solution and $\pi$ is a $2 \delta^\opt$-certificate, then we output the $x$. Otherwise, we double the value of $\hat \kappa$ and restart this procedure.

Assuming $\hat \kappa \geq \bar \kappa(\cX_A)$, 
 $\solveLP(A, b, c, u, \delta^\feas, \delta^\opt)$ returns a $\delta^\feas$-feasible and $\delta^\opt$-optimal solution along with a $2 \delta^\opt$-certificate. These provide a $\delta$-feasible and $\delta$-optimal solution, along with a dual certificate of $\delta$-optimality, in accordance with Lemma~\ref{lem:delta-cert}(i).

For the running time, $\solveLP(A, b, c, u, \delta^\feas, \delta^\opt)$ {makes}  at most $\log_2(n\|u\|_1/\delta^\opt)$ recursive calls (Theorem~\ref{thm:outer-main}), and, for each recursive call, $\getPair$ uses at most $O\left( n^{1.5} m^2  \|A\|_1^2 {\hat \kappa}^3 \right.$ $\log^2  \left.\left(\|u\|_1 n m  {\hat \kappa}/\deltast\right)\right)$ iterations (Theorem~\ref{thm:getpair-main}).
Note that $\hat \kappa$ will stop doubling no later than the first time $\hat \kappa \geq \bar \kappa(\cX_A)$. Therefore, the total number of iterations is at most $O\big( n^{1.5} m^2  \|A\|_1^2 \cdot{\bar \kappa(\cX_A)}^3 \log^3 \left(\|u\|_1 n m \|A\|_1 \cdot {\bar \kappa(\cX_A)}/\delta\big)\right)$ as $\deltast \geq \delta^\feas / n$.

\medskip
We now remove the feasibility assumption. We  use a two-stage approach,  similarly to Simplex. We consider the following extended system; $\mathbf{1}\in\R^m$ denotes the all 1's vector.
\[
\begin{aligned}
\min \pr{\mathbf{1}}{s'}+&\pr{\mathbf{1}}{s''}\\
Ax+s'-s''&=b\, \\
\bO\le x&\le u\\
\bO\le s',s''&\le  \|b\|_\infty
\end{aligned}
\] 
This system is trivially feasible with the solution $s'_i=\max\{b_i,0\}$ and $s''_i=\max\{-b_i,0\}$. Moreover, denoting the constraint matrix as $B=(A\, |\, I_m\, |\, -I_m)$, note that $\kappa(\cX_B)=\kappa(\cX_A)$ and $\bar\kappa(\cX_B)=\bar\kappa(\cX_A)$.

We obtain a $\delta / 4$-feasible and $\delta / 4$-optimal solution $(\xbar,s',s'' )$ for this system by applying $\solveLP$.
If the original system was feasible, then $(\xbar,s',s'' )$ provides a solution to the new LP with objective value of at most  $\delta / 4$. As the new cost vector is the all ones vector, we see that $\|s'\|_1+\|s''\|_1 < \delta/4$. As we show next, this implies that the returned solution $\bar x$ will be $\delta / 2$-feasible for the original system:
$$\|A \bar x - b\|_1 \leq \|A \bar x + s' - s''-b\|_1 + \|s'\|_1 + \|s''\|_1 \leq (\delta / 4)\cdot \|A\|_1  + \delta / 4 \leq (\delta / 2)\cdot \|A\|_1.$$
Both the second and third inequalities use $\|A\|_1 \geq 1$.
We now run the above algorithm with the modified right hand side $\bar b=A\bar x$. Now, $\LP(A,\bar b,c,u)$ is feasible, and a $\delta^\feas$-feasible solution for this system will also be $\delta$-feasible for the original system, since $\delta^\feas + \delta / 2 <\delta$.
\end{proof}

\section{Proofs of the Proximity Statements}\label{sec:prox-proof}

In this section, we prove Lemma~\ref{lem:close_opt_sol-kappa} and Theorem~\ref{thm:variable-fixing}.

\begin{proof}[Proof of Lemma~\ref{lem:close_opt_sol-kappa}]
We define the following modified capacities. For $i\in [n]$, let 
\[
\begin{aligned}
\bar\ell_{i}&:=\begin{cases}0&\mbox{if }c_{i} - A_i^\top \pi\le \sigma\, ,\\
x_{i}&\mbox{if }c_{i} - A_i^\top \pi> \sigma\, .
\end{cases}\\
\bar u_{i}&:=\begin{cases} u_{i}&\mbox{if }c_{i} - A_i^\top \pi\ge -\sigma\, ,\\
x_{i}&\mbox{if }c_{i} - A_i^\top \pi< -\sigma\, .
\end{cases}
\end{aligned}
\]
We now consider the optimization problem
\begin{equation}\label{eq:modified-lp}
\begin{aligned}
\min\ & \pr{c}{x}\\
Ax&=b\\
\bar\ell\le x&\le \bar u \, .
\end{aligned}
\end{equation}
This problem is feasible since $x$ is a feasible solution; let $\bar x$ be an optimal solution.

\begin{claim}
$\bar x_i=x_i$ for every $i\in \bJ(\pi,\sigma)$. 
\end{claim}
\begin{proof}
Consider a generalized path-circuit decomposition $\bar x-x=\sum_{k=1}^h g^k$.
Since $A(\bar x-x)=\bO$, any conformal circuit decomposition may contain only inner vectors: $Ag^k=\bO$ for all $k\le h$. We claim that $\supp(g^k)\cap \bJ(\pi,\sigma)=\emptyset$ for all $k\in[h]$; this implies the statement.
For a contradiction, assume there exists a $j\in \supp(g^k)\cap \bJ(\pi,\sigma)$ for some $k\in[h]$.

The optimality of $\bar x$ implies that $0\ge\pr{c}{g^k}=\pr{c-A^\top \pi}{g^k}$, as $Ag^k=\bO$.
Also, as we argue next, the definition of $\bar\ell$ implies that $g^k_i\ge 0$ for all $i\in [n]$ with  $c_i-A_i^\top \pi>\sigma$; for in this case, for every such index $i$, by definition, $\bar \ell_i=x_i$, and thus $\bar x_i-x_i\ge 0$, that is, $g^k_i\ge 0$. Similarly,
$g^k_i\le 0$ for all $i\in [n]$ with  $c_i-A_i^\top \pi<-\sigma$. Thus $(c_i-A_i^\top \pi)g^k_i\ge 0$ whenever $|c_i-A_i^\top \pi|>\sigma$.

Let $S\subseteq \supp(g^k)$ denote the set of indices $i$ with 
$(c_i-A_i^\top \pi)g^k_i<0$, and
let $\alpha:=\|g^k_S\|_\infty$. By the above, $j\notin S$ , and 
$(c_i-A_i^\top \pi)g^k_i\ge -\sigma\alpha$ for all $i\in S$.

The definition of $\kappa(A)\le \kappa(\cX_A)$ implies that
$|g^k_j|\ge \alpha/\kappa(\cX_A)$.
Hence, 
\[
\begin{aligned}
0&\ge\pr{c}{g^k}=\pr{c-A^\top \pi}{g^k}\\
&=\sum_{i\in \supp(g^k)\setminus S}(c_i-A_i^\top \pi)g^k_i+\sum_{i\in S}(c_i-A_i^\top \pi)g^k_i\\
&\ge (c_j-A_j^\top \pi)g^k_j-(n-1)\sigma\alpha\\
&\ge n\cdot\kappa(\cX_A)\cdot\sigma\cdot\frac{\alpha}{\kappa(\cX_A)}-(n-1)\sigma\alpha>0\, ,
\end{aligned}
\]
a contradiction.
\end{proof}

 The following claim completes the proof of \Cref{lem:close_opt_sol-kappa}.
\begin{claim}
Consider an optimal solution $x^*$ for \ref{eq:LP} with 
 $\|\bar x-x^*\|_1$ minimal. Then,
$\|\bar x-x^*\|_\infty\le \kappa(\cX_A)\cdot\theta(x,\pi,\sigma)$.
\end{claim}
\begin{proof}
Consider a generalized path-circuit decomposition $x^*-\bar x=\sum_{k=1}^h g^k$. This may only contain inner vectors, since $A(x^*-\xbar)=\bO$. By the choice of $x^*$, we must have $\pr{c}{g^k}<0$ for every $k$. Hence, $\bar x+\lambda g^k$ is not feasible for \eqref{eq:modified-lp} for any $\lambda>0$, as otherwise we get a contradiction to the optimality of $\bar x$ for \eqref{eq:modified-lp}. 

Therefore, for each $k$, there exists an $i\in \supp(g^k)$ with $\bar\ell_i=\bar x_i>0$ or $\bar u_i=\bar x_i<u_i$. 
By construction, for these $i$, $\xbar_i = x_i$.
Using the definition of $\kappa(A)\le \kappa(\cX_A)$ and the comformity of the decomposition, we get 
\begin{align*}
\|\xbar-x^*\|_\infty \le \max_j \sum_{k=1}^h |g_j^k| &\le 
\sum_{\substack{k=1:\\ 
                \text{some }i_k \in\supp(g^k) \text{ s.t.}\\ 
                \bar\ell_{i_k}=x_{i_k}>0\text{ or } 
                \bar u_{i_k}=x_{i_k}<u_{i_k}
               }
      }^h 
\kappa(A)\cdot  |g_{i_k}^k|  \\
&\le \kappa(A) \Big[\sum_{c_i -A_i^\top\pi \ge \sigma} x_i + \sum_{c_i -A_i^\top \pi < -\sigma}(u_i-x_i)\Big]\\
&\le \kappa(\cX_A)\cdot\theta(x,\pi,\sigma),
\end{align*}
which is
the claimed bound. 
\end{proof}
The lemma follows from combining the previous two claims.
\end{proof}

The next lemma bounds the difference in the optimum value if the right hand side changes; we will use it in the proof of Theorem~\ref{thm:variable-fixing}.

\begin{lemma}\label{lem:optval-prox}
 Let $A\in\R^{m\times n}$,  $c,u\in\R^n$, and $b,\bar b\in\R^m$. If both LPs are feasible, then
 \[
\big|\optv(A,b,c,u)-\optv(A,\bar b,c,u)\big|\le \kappa(\cX_A)\cdot\|c\|_1\cdot \|b-\bar b\|_1\, .
 \]
\end{lemma}
\begin{proof}
Let $x$ be an optimal solution to \ref{eq:LP} and $\bar x$ an optimal solution to $\LP(A,\bar b,c,u)$ such that $\|x-\bar x\|_1$ is minimal, and consider a generalized path-circuit decomposition $\bar x-x=\sum_{k=1}^h g^k$. By the choice of $x$ and $\bar x$, the decomposition contains no inner circuits. 
Thus, by Lemma~\ref{lem:contig-proximity}, $\|x-\bar x\|_\infty\le \kappa(\cX_A)\cdot\|b-\bar b\|_1$. Therefore,
\[
\big|\optv(A,b,c,u)-\optv(A,\bar b,c,u)\big|=\big|\pr{c}{x-\bar x}\big|\le \|c\|_1\cdot \|x-\bar x\|_\infty\le  \kappa(\cX_A)\cdot\|c\|_1 \cdot \|b-\bar b\|_1\, ,
\]
proving the claim.
\end{proof}

\begin{proof}[Proof of Theorem~\ref{thm:variable-fixing}]
Feasibility is trivial, since $\xN$ is a feasible solution.
By the definition of $\theta(x,\pi,\sigma)$, 

\[
\begin{aligned}
\|b-\bar b-\AJtwo {\uJtwo}\|_1&\le \|b-Ax\|_1+ \|Ax-\AN \xN-\AJtwo {\uJtwo}\|_1\\
&={}\|b-Ax\|_1+\|\AJone \xJone-\AJtwo (\uJtwo-\xJtwo)\|_1\\
&\le{} \theta(x,\pi,\sigma)\cdot\|A\|_1+\|Ax-b\|_1\, ,
\end{aligned}
\]
which finishes the proof of the feasibility condition.

To show the optimality condition, we will apply \Cref{lem:close_opt_sol-kappa,lem:optval-prox} as follows.
Consider the linear program $\LP(A,Ax,c,u)$, for which $x$ is a feasible solution. Let $x^*$ be the optimal solution to this LP given by \Cref{lem:close_opt_sol-kappa}. Defining $y^*\coloneqq x^*|_{\scriptscriptstyle{\!N}}$ and $b^*\coloneqq \AN y^*$, we obtain that $y^*$ is an optimal solution to $\LP(\AN,b^*,\cN,\uN)$.

For convenience, let $\rho\coloneqq2\cdot\kappa(\cX_A)\cdot \theta(x,\pi,\sigma)$, $J\coloneqq J_1\cup J_2$, and $k\coloneqq|J|$.
For each $i\in J_2$, we have that $x^*_i\leq x_i + \kappa(\cX_A)\cdot \theta(x,\pi,\sigma)\leq \rho$, where the first inequality follows from \Cref{lem:close_opt_sol-kappa} and the second one by the definition of $\theta(x,\pi,\sigma)$. Similarly, for each $i\in J_1$, we obtain $x^*_i\ge u_i -\rho$. This yields the bound
\[
\big|\optv(\AN,b^*,\cN,\uN)+\pr{\cJtwo}{\uJtwo}-\optv(A,Ax,c,u)\big|=
\big|\pr{\cN}{y^*}+\pr{\cJtwo}{\uJtwo}-\pr{c}{x^*}\big|\le k\rho \|c\|_\infty\, . \numberthis \label{ineq:linear-program-diff-iteration-1}
\]
To finish the proof we need to bound the differences $|\optv(\AN,\bar b,\cN,u_N)-\optv(\AN,b^*,\cN,\uN)|$ and $|\optv(A,b,c,u)-\optv(A,Ax,c,u)|$, which we will do by applying \Cref{lem:optval-prox}. This yields
\begin{align}
\big|\optv(A,b,c,u)-\optv(A,Ax,c,u)\big|\leq \kappa(\cX_A)\cdot \|c\|_1 \cdot\|Ax-b\|_1 \label{ineq:b-minus-Ax-bound}
\end{align}
and
\begin{align}
&\big|\optv(\AN,\bar b,\cN,\uN)-\optv(\AN,b^*,\cN,\uN)\big|\le \kappa(\cX_A)\cdot\|c\|_1\cdot\|\bar b-b^*\|_1. \label{ineq:linear-program-diff-iteration-2}
\end{align}

To use \eqref{ineq:linear-program-diff-iteration-2}, we further bound $\|\bar b-b^*\|_1$ as follows:
\begin{align*}
    \|\bar b-b^*\|_1& = \|\AN\xN - \AN y^*\|_1 \\
    &= \|Ax - Ax^* - \AJ(\xJ-\xJ^*)\|_1 \\
    &= \|\AJ(\xJ-\xJ^*)\|_1 \\
    &\leq k\cdot \|A\|_1 \cdot \kappa(\cX_A) \cdot \theta(x,\pi,\sigma), \numberthis \label{ineq:linear-program-diff-iteration-3}
\end{align*}
where the inequality in the last line follows by how we obtained $x^*$ from \Cref{lem:close_opt_sol-kappa}.
Finally, we conclude by combining \eqref{ineq:linear-program-diff-iteration-1}, \eqref{ineq:b-minus-Ax-bound}, \eqref{ineq:linear-program-diff-iteration-2}, and \eqref{ineq:linear-program-diff-iteration-3}:
\begin{align*}
    &\left|\optv(\AN,\bar b,\cN,\uN)+\pr{\cJtwo}{\uJtwo}-\optv(A,b,c,u)\right|\\
    &~~~~ \le \left|\optv(\AN,b^*,\cN,\uN)+\pr{\cJtwo}{\uJtwo}-\optv(A,Ax,c,u)\right| + \big|\optv(A,Ax,c,u) - \optv(A,b,c,u)\big| \\
    &~~~~~~~+ \left|\optv(\AN,\bar b,\cN,\uN)-\optv(\AN,b^*,\cN,\uN)\right| \\
    &~~~~ \le k\cdot\rho\cdot \|c\|_{\infty}  + \kappa(\cX_A)\cdot\|c\|_1 \cdot\|Ax-b\|_1 + \kappa(\cX_A)\cdot\|c\|_1\cdot\|\bar b-b^*\|_1\\
    &~~~~ \le k\cdot\rho\cdot \|c\|_{\infty}  + \kappa(\cX_A)\cdot\|c\|_1 \cdot\|Ax-b\|_1  + k\cdot\kappa(\cX_A)^2\cdot\|c\|_1\cdot \|A\|_1 \cdot\theta(x,\pi,\sigma)\\
    &~~~~ = \kappa(\cX_A)\cdot\|c\|_1 \cdot\|Ax-b\|_1 +
    \big[2k\cdot\kappa(\cX_A)\cdot\|c\|_{\infty}  + k\cdot\kappa(\cX_A)^2\cdot\|c\|_1\cdot \|A\|_1\big] \cdot\theta(x,\pi,\sigma)\\
    &~~~~ \le \kappa(\cX_A)\cdot\|c\|_1 \cdot\|Ax-b\|_1 + k\cdot\kappa(\cX_A)\cdot\|c\|_{1}\cdot
    \big(2 + \kappa(\cX_A)\cdot \|A\|_1\big) \cdot\theta(x,\pi,\sigma).
\end{align*}

This proves the optimality condition. The third statement, that is, the cost reduction, directly follows from the definitions of $J_1$, $J_2$,  $N$ and $\sigma$.

\end{proof}

\section{Analysis of the outer loop}\label{sec:analysis-outer}

\begin{proof}[Proof of Theorem~\ref{thm:outer-main}]
    We first show that the algorithm terminates with recursion depth at most $\log_2(n\|u\|_1/\delta^\opt)$. By the cost reduction condition in \Cref{thm:variable-fixing}, $\lambda\geq2$ holds in each recursive call in  \Cref{alg:outermost}. If $\delta_0^\opt$ denotes the initial $\delta^\opt$ and $\delta_j^\opt$ denotes the $\delta^\opt$ of the $j$-th recursive call, then it follows that $\delta_j^\opt\geq 2^j \delta_0^\opt /n$. Thus, {when}  $j\geq \log_2(n\|u\|_1/\delta_0^\opt)$, we have  $\delta_j^\opt\geq \|u\|_1$ and the algorithm terminates.

    Now we prove the correctness of the algorithm by induction on the recursion depth. For the induction start, we assume that $\delta^\opt\geq \|u\|_1$, in which case an arbitrary $\delta^\feas$-feasible solution $x$ is returned. Let $x^*$ be an optimal solution to \ref{eq:LP}. Since $\bO\leq x,x^*\leq u$, it follows that $\pr{c}{x}-\optv(A,b,c,u) = \pr{c}{x-x^*}\leq \|u\|_1\|c\|_\infty\leq \delta^\opt\|c\|_\infty$, proving $\delta^\opt$-optimality.

    For the inductive step, assume $\delta^\opt< \|u\|_1$. Then the algorithm invokes the subroutine \getPair{} {followed by}  a recursive call to itself.
   Recall that the subroutine \getPair{} returns 
   $(x,\pi)$ that satisfy
    the three bounds stated as part of the subroutine.
    We distinguish two cases depending on whether $J_1\cup J_2$ is empty or not.
    
    If $J_1\cup J_2=\emptyset$, then the induction hypothesis applied to the call in line~\ref{line:nofixing} implies that $x^\outindex$ is $\delta^\feas$-feasible for the original system. 
    Moreover, it is a $\lambda\delta^\opt$-optimal solution to $\LP(A,b,c^\new,u)$, with
    $\lambda = \|c\|_\infty/(2\|c^\new\|_\infty)$ in this case.
    Note that $\optv(A,b,c^\new,u)=\optv(A,b,c,u)-\pr{\pi}{b}$. This implies:
    \begin{align*}
        \lvert\pr{c}{x^\outindex}-\optv(A,b,c,u)\rvert
        &\leq\lvert\pr{c-c^\new}{x^\outindex}-\pr{\pi}{b}\rvert+\lvert\pr{c^\new}{x^\outindex}-\optv(A,b,c^\new,u)\lvert\\
        &\leq\lvert\pr{\pi}{Ax^\outindex-b}\rvert+\gain\delta^\opt\|c^\new\|_\infty{~~~~\text{(as $x^\outindex$ is $\lambda\delta^\opt$-optimal)}}\\
            &\leq \|\pi\|_{\infty} \cdot\delta^\feas \|A\|_1+\gain\delta^\opt\|c^\new\|_\infty~~~~\text{(as $x^\outindex$ is $\delta^\feas$-feasible)}\\
            &\leq \frac{\delta^\opt\|c\|_\infty}{2\delta^\feas\|A\|_1} \cdot \delta^\feas \|A\|_1+\gain\delta^\opt\|c^\new\|_\infty\\
            &\hspace*{1in}\text{(using the {upper} bound on $\|\pi\|_{\infty}$ in \getPair{}}\\
            &\hspace*{1in}\text{{~and the lower bound on $\deltaopt/\deltafeas$ in \solveLP)}}\\
        &\leq \frac{\delta^\opt}{2}\|c\|_\infty+\frac{\delta^\opt}{2}\|c\|_\infty = \delta^\opt \|c\|_\infty,
    \end{align*}
    finishing the first case. 
    
        Otherwise, if $J_1\cup J_2\neq\emptyset$, note that $\LP(\AN,\bar b, \cN^\new, \uN)$ is feasible because $\xN$ is a feasible solution. Thus we can apply the induction hypothesis to the recursive call and obtain that $\xN^\outindex$ is $(\delta^\feas\cdot |N|/n)$-feasible and $(\gain\delta^\opt\cdot |N|/n)$-optimal for  $\LP(\AN,\bar b, \cN^\new, \uN)$. We combine this with \Cref{thm:variable-fixing} and the properties of the subroutine to prove first {the} $\delta^\feas$-feasiblity and then {the} $\delta^\opt$-optimality of $x^\outindex$ for the original system.

    With respect to feasiblity, we obtain:
    \begin{align*}
        \| A x^\outindex - b \|_1&\leq\|\AN \xN^\outindex - \bar b\|_1 + \|\bar b+\AJtwo \uJtwo-b\|_1~~~~\text{(as $A x^\outindex = \AN \xN^\outindex + \AJtwo \uJtwo$)}\\ 
        &\leq (\delta^\feas\cdot |N|/n)\|\AN\|_1 +\theta(x,\pi,\sigma)\cdot \|A\|_1 + \|Ax-b\|_1~~~~\text{(explanation below)}\\
        &\leq (\delta^\feas\cdot |N|/n)\|\AN\|_1 + \delta^\feas\|A\|_1/n~~~~~~~~~~~~~~~~\text{(explanation below)}\\
        &\leq \delta^\feas\|A\|_1.
    \end{align*}
    The second inequality uses induction for the first term and \Cref{thm:variable-fixing} for the second term.
    The third inequality follows because \getPair{} returns a solution bounding the second {and third} terms in the second line, the RHS of \eqref{eqn::feasibility-bound}, by the second term in the third line.
    
    To prove approximate optimality, we first focus on the coordinates in $N$ and obtain the following bound. (See below for explanations.)
    \begin{align}
        \lvert \pr{\cN}{\xN^\outindex}- &\optv(\AN,\bar b,\cN^\new,\uN) - \pr{\pi}{\bar b}\rvert \nonumber\\
        &\leq \lvert\pr{\cN-\cN^\new}{\xN^\outindex}-\pr{\pi}{\bar b}\rvert+\lvert\pr{\cN^\new}{\xN^\outindex}-\optv(\AN,\bar b,\cN^\new,\uN)\lvert \nonumber\\
        &\leq\lvert\pr{\pi}{\AN\xN^\outindex-\bar b}\rvert+(\gain\delta^\opt\cdot |N|/n)\cdot|c^\new\|_\infty \nonumber\\
        &\leq \|\pi\|_\infty \|\AN\|_1 \delta^\feas\cdot \frac{|N|}{n}+ \frac{\delta^\opt\lvert N \rvert}{2n}\cdot\|c\|_\infty \nonumber\\
         &\leq \frac{\delta^\opt \|c\|_\infty}{2\delta^\feas\|A\|_1} \|\AN\|_1 \delta^\feas\cdot \frac{|N|}{n} + \frac{\delta^\opt\lvert N \rvert}{2n}\|c\|_\infty \nonumber\\
        &\leq (\delta^\opt\cdot |N|/n)\|c\|_\infty. \label{eqn::Ncoord-cost-bound}
    \end{align}
   The bound $\|\pi\|_\infty \|\AN\|_1 \delta^\feas\cdot |N|/n$ on the first term in the {third}  line follows as $\xN^\outindex$ is a $\delta^\feas |N| / n$ feasible solution.  In turn, the bound on this term follows because of the bound on $\|\pi\|_\infty$ in \getPair{} and $0<\delta^\feas (8 n \sqrt{m} \cdot\kappa(\cX_A) \|A\|_1 )\le\delta^\opt \le \|u\|_1$ {(the second inequality is an input constraint for \solveLP, and the final inequality arises due to computing $N$ in \solveLP{} only when $\deltaopt<\|u\|_1$)}.
    
       Finally, we obtain:    
    \begin{align*}
      \lvert\pr{c}{x^\outindex}-\optv(A,b,c,u)\rvert&\leq \lvert \pr{\cN}{\xN^\outindex}- \optv(\AN,\bar b,\cN^\new,\uN) - \pr{\pi}{\bar b}\rvert \\
        &~~~~~~~~~~~~~+ \lvert \optv(\AN,\bar b,\cN^\new,\uN) + \pr{\pi}{\bar b}+\pr{\cJtwo}{\uJtwo}-\optv(A,b,c,u)\rvert\\
        &\leq (\delta^\opt\cdot |N|/n)\|c\|_\infty + \delta^\opt\|c\|_\infty/n\\
        &\leq \delta^\opt\|c\|_\infty.
    \end{align*}
   To obtain the first inequality, recall that $x^\outindex$ is a combination of $\xN^\outindex$, $\xJone^\outindex=0$, and $\xJone^\outindex=\uJtwo$. For the next inequality, note that the {first} term on line {2}  is bounded by $(\delta^\opt\cdot |N|/n)\|c\|_\infty$ using \eqref{eqn::Ncoord-cost-bound}, and the second term is bounded by the LHS of \eqref{eqn::optimality-bound}, which is bounded by $\delta^\opt\|c\|_\infty/n$ according to the second condition of \getPair{}.
 This concludes the second case.  
 \end{proof}

\section{Analysis of the inner loop} \label{sec:analysis-inner}

In this section, we prove Theorem~\ref{thm:getpair-main} on the correctness and running time of the inner loop.
The proof uses three main lemmas. The first one bounds the number of iterations in \RFGM{}, using Theorem~\ref{thm:r-fgm-convergence-rate}. This lemma will be proved in  Section~\ref{subsed:R-FEM-pf-convergence}.
{Recall that the approximation factor $\varepsilon$ was chosen to be  $\varepsilon = 1 /(8n\cdot\lceil\kappa(\cX_{\oriA})\rceil)$}.

\begin{restatable}{lemma}{RFGMconvergence}\label{lem:R-FEM-pf-convergence}
    After $O\left(k \sqrt{n} m^2 \|A\|_1^2 \cdot \bar\kappa^2(\cX_{A}) / \varepsilon \right)$ iterations, \RFGM{} returns an $\ee^{-2k} \|b \|_2^2 / (2 \|A\|_1^2)$-approximate minimizer of $\Ft{\tau}$. 
\end{restatable}
The next lemma strengthens Proposition~\ref{prop:Fct-basic}, by asserting a Lipschitz property of $\Ftopt{\tau}$ as a function of $\tau$. 
\begin{lemma} \label{lem:range-tau}
    $\Ftopt{\tau}$ is a non-increasing and continuous function of $\tau$. In addition,
    If $\Ftopt{\tau} \le \paraBarK ^2 \targetaccuracy$, then $\Ftopt{\tau - \paraBarK  \sqrt{\targetaccuracy}/2} \leq (2\paraBarK ^2 - 1) \targetaccuracy$. 
\end{lemma}
\begin{proof} The first part is simple and was already stated as Proposition~\ref{prop:Fct-basic}. Let $x^*$ be the optimal solution of $\Fcepstauopt$, and 
    let $\Delta:= \paraBarK  \sqrt{\zeta}/2 $. The lemma follows by showing that 
\begin{equation}\label{eq:tau-Delta}
F_{ \tau - \Delta}(x^*)-\Fcepstauopt\le (\paraBarK ^2 - 1) \targetaccuracy
\end{equation}
    Since $\Fcepstauopt \le \paraBarK ^2 \zeta$ and $\Fcepstauopt \geq  (\max\{0, \langle \ceps, x^* \rangle - \tau \} )^2 / 2$, we have $\max\{0, \langle \ceps, x^* \rangle - \tau \} \leq \paraBarK   \sqrt{2 \zeta}  = 2 \sqrt{2} \Delta$. 
    {Therefore,} 
    \begin{align*}
    F_{ \tau - \Delta}(x^*) - \Fcepstauopt &= (\max\{0, \langle \ceps, x^* \rangle - \tau + \Delta\} )^2 / 2 - (\max\{0, \langle \ceps, x^* \rangle - \tau \} )^2 / 2
    \\
    &\leq \max\{0, \langle \ceps, x^* \rangle - \tau \} \cdot\Delta+ \frac{1}{2} \Delta^2\,\\
    &{\le 2\sqrt2 \Delta \cdot \Delta + \frac 12 \Delta^2.}
    \end{align*} 
    Using the {above} bound, we get \eqref{eq:tau-Delta} from
    \begin{align*}
       &F_{ \tau - \Delta}(x^*) - \Fcepstauopt
        \leq  2 \sqrt{2} \Delta^2 + \frac{1}{2} \Delta^2  \\
         &=  (1 / \sqrt{2} + 1 / 8) \paraBarK ^2 \zeta  \leq (\paraBarK ^2 - 1)\zeta\, .
    \end{align*}
    The last inequality follows by Lemma~\ref{lem:K-lower}.
    \end{proof}

The third lemma shows that if $\Ft{\tau}(x)$ {lies in}  the appropriate interval, then the output satisfies the desired properties of the algorithm. The proof is given in Section~\ref{subsec:linear-program-upper-theta-delta}.

\begin{lemma}\label{lem:linear-program-upper-theta-delta}
     Let $x$ be a $\zeta$-approximate minimizer and $\pi$ as defined in the algorithm. Assume that  $\Ft{\tau}(x) \in [\paraBarK ^2 \zeta, 2 \paraBarK ^2 \zeta]$. Then  
     \begin{enumerate}[(i)]
        \item $ \theta(x, \pi, \sigma) \leq n \sqrt{\zeta} / 2$, 
         \item $ \|Ax - b\|_1 \leq    128 n^2 m \cdot \kappa^2(\cX_A) \cdot {\|A\|_1^2} \sqrt{\zeta}$, and
         \item $\|\pi\|_{\infty} \leq 4n \sqrt{m} \cdot \kappa(\cX_A)\cdot  \|c\|_{\infty} $.
     \end{enumerate} 
\end{lemma}

We now give the proof of Theorem~\ref{thm:getpair-main} using these three lemmas.
\begin{proof}[Proof of Theorem~\ref{thm:getpair-main}]
First of all, we need to show that the algorithm eventually outputs a primal dual pair $(x, \pi)$, i.e., the condition 
$\Ft{\tau}(x) \in [\paraBarK ^2 \targetaccuracy, 2 \paraBarK ^2 \targetaccuracy]$ will be met, and the claimed bound on the number of calls to R-FGM is applicable. This follows by the next claim.

\begin{claim}
$\tau^{+} - \tau^{-} \ge\paraBarK \sqrt{ \targetaccuracy} / 2$ throughout Algorithm~\ref{alg:getPair}.
\end{claim}
\begin{proof}
    We show the following invariant property of the algorithm: $$F^*_{\tau^-} > (2 \paraBarK ^2 - 1) \targetaccuracy \text{~and~} F^*_{ \tau^+} < \paraBarK ^2 \targetaccuracy.$$ 
    Initially, for any feasible $x$ to \ref{eq:LP}, we have $F_{\tau^+}(x) =0$, and therefore $F^*_{ \tau^+} = 0$. Also,
    $F^*_{\tau^-} \geq 2 \paraBarK ^2  \targetaccuracy $ as $F_{\tau^{-}} \geq \frac{1}{2}(\max\{0, \langle \ceps, x \rangle  - \tau^{-}\})^2 \geq \frac{1}{2 }( 2 \paraBarK \sqrt{ \targetaccuracy})^2 = 2 \paraBarK ^2  \targetaccuracy$, noting that $\pr{\ceps}{x}\ge -\|\ceps\|_\infty\|u\|_1\ge-\|u_1\|$.
    During the algorithm, each computed $x$ is a $\zeta$-approximate minimizer of $\Fctau$, and therefore the updating of
    $\tau^-$ and $\tau^+$ in steps 5--8 maintains the invariant.
    
    For a contradiction, assume $\tau^{+} - \tau^{-} \leq \paraBarK \sqrt{ \targetaccuracy} / 2$ at some point. Then, 
    $F^*_{\tau^+} < \paraBarK ^2 \targetaccuracy$ and Lemma~\ref{lem:range-tau} imply 
    $F^*_{ \tau^-} \le (2 \paraBarK ^2 - 1) \targetaccuracy$, a contradiction.
    \end{proof}
    
To bound the number of calls to  R-FGM, note that initially $\tau^+-\tau^-=2\|u\|_1+2\paraBarK \sqrt{ \targetaccuracy} \leq 4 \|u\|_1$ as $\|u\|_1 \geq {\deltafeas}$, as required for the input of $\getPair$; and $\tau^+-\tau^-$ is halved in every iteration. According to the above claim, it never goes below $\paraBarK \sqrt{ \targetaccuracy} / 2$, {and applying  \Cref{lem:K-lower} yields the following bound on the number of calls:} 
\[
O\left(\log\left({\|u\|_1}/{(\paraBarK \sqrt{ \targetaccuracy})}\right)\right)=O\left(\log\left(\|u\|_1 nm\cdot\kappa(\cX_{A})/\deltast\right)\right)\, .
\]
    In  each R-FGM call, by Lemma~\ref{lem:R-FEM-pf-convergence}, to obtain a $\targetaccuracy$-approximate minimizer takes at most $O(k {\sqrt n m^2} \|A\|_1^2\cdot \bar\kappa^2(\cX_{\oriA}) / \varepsilon )$ steps, where $k =  O\left(\log (\|b\|_2^2/(\|A\|_1^2 \targetaccuracy)\right)$ and $\varepsilon = 1 / [8n \cdot\kappa(\cX_{\oriA})]$. The claimed bound follows as $\|b\|_2^2 \leq \|b\|_1^2 \leq \|A\|^2_1 \cdot \|u\|^2_1$, which follows from the feasibility assumption.

The bounds on $x$ and $\pi$ asserted in the theorem follow by simple calculation from the bounds in Lemma~\ref{lem:linear-program-upper-theta-delta} and the definition of $\targetaccuracy$ in \eqref{eq:def-para}.
\end{proof}

\subsection{Optimality bounds from the potential function} \label{subsec:linear-program-upper-theta-delta}

Recall the parameters $\varepsilon = 1 /(8n\cdot\lceil\kappa(\cX_{\oriA})\rceil)$, $\sigma=\|c\|_\infty/(4n\cdot\lceil\kappa(\cX_A)\rceil)=2\varepsilon\|c\|_\infty$. 
Recall that in Algorithm~\ref{alg:getPair}, we defined
\[
\alphaxtau(x):= \alphaxtau= \max\{0, \langle \ceps, x \rangle - \tau\}\, ,\quad\mbox{and}\quad
\pi(x) := \pi= \frac{ \|c\|_{\infty}}{\|A\|^2_1\alphaxtau }(b-Ax)\, \]
Let us further define
\begin{equation}
\label{eq:beta-def}
\beta(x):=\frac{\|Ax-b\|_2}{\|A\|_1}\, .
\end{equation}
With this notation, we can write
\begin{align*}
    \Fcepstau(x) = \frac{1}{2}\alpha(x)^2 + \frac{1}{2}\beta(x)^2\, .
\end{align*}
We will also often use the gradient, {which}  can be expressed as 
\begin{equation}\label{eq:grad}
\nabla \Fcepstau(x)=\alpha(x)\left(\hat c-\frac{A^\top \pi(x)}{\|c\|_\infty}\right) \, .
\end{equation}
When $x$ is clear from the context, we simply write $\alpha$, $\beta$, and $\pi$.
We will use a simple convexity statement, formulated in the following general form.

\begin{proposition}\label{prop:smooth-bound}
Let $f\, :\, \R^n\to \R$ be a continuously differentiable convex function with $[\bO,u]\subseteq \mathrm{dom}(f)$, such that for some $M\ge 1$, $f$ satisfies the following smoothness property: for {every}  index $i\in [n]$, {for every pair}  $x,y\in \R^n$ such that $x_j=y_j$ for {all} $j\neq i$, 
\begin{equation}\label{eq:smoothness}
|\nabla_i f(x)-\nabla_i f(y)| \le M|x_i-y_i|\, .
\end{equation}
{Also, suppose}  that for some $\eta>0$, $x$ is an $\eta$-approximate minimizer to the program
\[
\min f(x)\quad \mbox{s.t.} \quad x\in[\bO,u]\, .
\]
Then, for any $i\in [n]$, the following hold:
\begin{enumerate}[(i)]
\item If $\nabla_i f(x)\ge 2M\sqrt{\eta}$, then $x_i\le \sqrt{\eta}/M$.
\item If $\nabla_i f(x)\le -2M\sqrt{\eta}$, then $x_i\ge u_i-\sqrt{\eta}/M$.
\end{enumerate}
\end{proposition}
\begin{proof}
We only prove part (i); part (ii) follows analogously.  For a contradiction, assume that for some $i\in [n]$,
$\nabla_i f(x)\ge 2M\sqrt{\eta}$ and $x_i>\sqrt{\eta}/M$.
Let us define $z\in \R^n$ with $z_i:=x_i-\sqrt{\eta}/M>0$
and $z_j:=x_j$ for $j\neq i$. Thus, $z\in [\bO,u]$, and by the smoothness property, $\nabla_i f(z)\ge \nabla_i f(x)-\sqrt{\eta}\ge (2M-1)\sqrt{\eta}$. By convexity, 
\[
f(x)\ge f(z)+\pr{\nabla f(z)}{x-z}= f(z)+\nabla_i f(z)\cdot\frac{\sqrt{\eta}}{M}\ge f(z)+(2M-1)\sqrt{\eta}\cdot\frac{\sqrt{\eta}}{M}>f(z)+\eta\, ,
\]
using  $M\ge 1$. This is a contradiction to the assumption that {$x$}  is an $\eta$-approximate minimizer.
\end{proof}

The next lemma is used to prove the first key ingredient of 
Lemma~\ref{lem:linear-program-upper-theta-delta}, 
 an upper bound on $\theta(x,  \pi, \sigma)$, provided that $\alpha(x)$ is sufficiently large.
 
\begin{lemma} \label{lem:linear-program-theta-upperbound}
Suppose $x$ is a $\zeta$-approximate minimizer of $\Fcepstau(x)$ satisfying $\varepsilon \cdot \alphaxtau(x)   \geq 4  \sqrt{\zeta}$, and $\pi=\pi(x)$.
Then
$\theta(x,  \pi, \sigma) \leq n\sqrt{ \zeta} / 2$.
\end{lemma}
\begin{proof}
We  use $\alpha=\alpha(x)$, $\beta=\beta(x)$, $\pi=\pi(x)$ throughout.
Recall that $\theta(x, \pi, \sigma)  = \sum_{i : c_i - A_i^\top \pi > \sigma} x_i + \sum_{i: c_i - A_i^\top \pi < - \sigma} (u_i - x_i)$, and that the vector $\ceps$ was obtained by normalizing as $c/\| c\|_\infty$, and then rounding down each positive entry to the nearest integer multiple of $\varepsilon$, and rounding up each negative entry to the nearest integer multiple of $\varepsilon$. Also noting that $\sigma=2\varepsilon\|c\|_\infty$, 
we obtain the following upper bound in terms of $\ceps$.
\begin{align*}
\theta(x,  \pi, \sigma ) \leq 
\sum_{i : \ceps_i - A_i^\top \pi  / \|c\|_{\infty} > \varepsilon} x_i + \sum_{i: \ceps_i - A_i^\top \pi  / \|c\|_{\infty} < -  \varepsilon} (u_i - x_i). \numberthis \label{ineq:theta-upper-bound-proof-of-lemma}
\end{align*}
We will show each of the terms $x_i$ and $u_i-x_i$ in these two sums is bounded by $\frac{\sqrt{\zeta}} {2}$,
 and then the {result}  is immediate. 
Recall from \eqref{eq:grad} that 
\[
\ceps_i - A_i^\top \pi / \|c\|_{\infty}=\nabla_i \Fcepstau(x) / \alphaxtau\, .
\]
We {apply}  Proposition~\ref{prop:smooth-bound} {to}  $\Fcepstau(x)$ and $\eta=\zeta$. From \eqref{eq:grad}, {a simple calculation shows}  that $f(x)=\Fcepstau(x)$ satisfies the smoothness bound \eqref{eq:smoothness} with $M=2$. Namely, for any $i\in [n]$ and for any $x,y\in \R^n$ such that $x_j=y_j$ for $i\neq j$,
\begin{align*}
    |\nabla_i \Fcepstau(x) - \nabla_i \Fcepstau(y)| &\leq \|\ceps\|_{\infty} \|x - y\|_1 + |A_i^\top A_i| \cdot \|(x - y)\|_{1} / \|A\|_1^2 \leq 2 |x_i - y_i|\, .
\end{align*}
 $\ceps_i - A_i^\top \pi  / \|c\|_{\infty} > \varepsilon$ is equivalent to $\nabla_i \Fcepstau(x)\ge \varepsilon\alpha$. By the assumption of the lemma, $\varepsilon\alpha\ge 4\sqrt{\zeta}=2M\sqrt{\zeta}$.

Thus, Proposition~\ref{prop:smooth-bound}(i) implies that whenever  $\ceps_i - A_i^\top \pi  / \|c\|_{\infty} > \varepsilon$, we must have $x_i\le \sqrt{\zeta}/2$. Similarly, Proposition~\ref{prop:smooth-bound}(ii) 
implies that whenever  $\ceps_i - A_i^\top \pi  / \|c\|_{\infty} < -  \varepsilon$, we have $u_i-x_i\le \sqrt{\zeta}/2$. This completes the proof.
\end{proof}

Note the previous lemma requires a lower bound on $\alphaxtau$. Our second lemma shows that this requirement is satisfied if one get a good approximate minimizer with a sufficiently large function value.

\begin{lemma}\label{lem:lower-bound-alpha-linear-system}
    Suppose $x$ is a $\zeta$-approximate minimizer of $\Fcepstau(x)$,  satisfying $\Fcepstau(x) \geq 10 \paraK ^2\zeta$.
    Then, $\alphaxtau(x)  \geq \frac{1}{2\paraK } \sqrt{\Fcepstau(x)}$.
\end{lemma}

The proof relies on the following statement:

\begin{proposition}\label{lem:lower-bound-alpha-alpha-beta-tradeoff}
Assume \ref{eq:LP} is feasible, and 
let $x\in [\bO,u]$ with  $\betaxtau(x) > 0$. Then, for any $0 \leq \mu \leq 1$, there exists a solution $x'\in [\bO,u]$ such that $\alphaxtau(x') \leq \alphaxtau(x) + {\mu} \paraK   \betaxtau(x)/n $ and $\betaxtau(x') \leq  (1 - \mu) \betaxtau(x)$. 
\end{proposition}
\begin{proof}
   Let $z\in\R^n$ be chosen as
   \[
   z:=\arg\min\{\|z-x\|^2\mid  Az=b\, ,\, z\in[\bO,u]\}\, \quad \mbox{and set}\quad x' := \mu z + (1 - \mu) x\, .
\]
    Clearly, $\betaxtau(x') = (1 - \mu) \betaxtau(x)$. By Lemma~\ref{lem:contig-proximity}, $\|z - x\|_{\infty} \leq \sqrt{m} \cdot\kappa({\cX_{A}})\cdot \| Ax - b\|_ 2$. Therefore, {$\alphaxtau$} increases by at most 
    \begin{align*}
    \langle \ceps, (\mu z + (1 - \mu) x)  - x \rangle &\leq\mu \|\ceps\|_1 \| z - x \|_{\infty} \leq  \mu \sqrt{m} \cdot\kappa({\cX_{\oriA}})\cdot \| Ax - b\|_ 2 \\
    &= \mu \sqrt{m} \cdot \kappa({\cX_{\oriA}})\cdot \|A\|_1 \betaxtau(x)
    \leq \mu  \paraK \betaxtau(x)/n \, ,
    \end{align*}
    recalling the definition of $\paraK$ in \eqref{eq:def-para}.
\end{proof}

\begin{proof}[Proof of Lemma~\ref{lem:lower-bound-alpha-linear-system}]
We distinguish two cases.

\paragraph{Case i: $2\paraK\alpha(x)\ge \beta(x)$.}
In this case, 
   \begin{align*}
        \Fcepstau(x) &= \frac{1}{2} \alphaxtau(x) ^2 + \frac{1}{2} \betaxtau(x)^2 \leq \frac{4\paraK^2+1}{2} \alphaxtau(x)^2 \, .
    \end{align*}
    Thus, the claimed $\alphaxtau(x)  \geq \frac{1}{2\paraK } \sqrt{\Fcepstau(x)}$ follows by recalling $\paraK\ge 1$ (Lemma~\ref{lem:K-lower}).

\paragraph{Case ii: $2\paraK\alpha(x)< \beta(x)$.}
Let us use Proposition~\ref{lem:lower-bound-alpha-alpha-beta-tradeoff} 
with 
\[
\mu := \frac{\betaxtau(x) -  \paraK  \alphaxtau(x)  }{(\paraK ^2  +1)\betaxtau(x) }\, .
\]
Clearly, $\mu\in [0,1]$. Thus, there exists 
 $x'\in[\bO,u]$ such that
$\alphaxtau(x') \leq \alphaxtau(x) + \mu \paraK  \betaxtau(x)$ and $\betaxtau(x') \leq  (1 - \mu) \betaxtau(x)$.
{Therefore,} 
   \begin{align*}
    \Fcepstau(x') &= \frac{1}{2} \alphaxtau(x')^2 + \frac{1}{2} \betaxtau(x')^2  \leq \frac{1}{2} (\alphaxtau(x)+   \mu \paraK  \betaxtau(x))^2 + \frac{1}{2} (1 - \mu)^2 \betaxtau(x)^2.
    \end{align*}
Above, we picked $\mu$ to minimize this expression, and some calculation yields 
\[
\Fcepstau(x') \leq \frac{ (\alphaxtau(x)  + \paraK  \betaxtau(x))^2}{2 (\paraK ^2 + 1)}.\, 
\]
 Further calculation {yields} 
\[
\Fcepstau(x) - \Fcepstau(x') \geq \frac{ (\paraK  \alphaxtau(x)  -  \betaxtau(x))^2}{2 (\paraK ^2 + 1)}> \frac{\betaxtau(x)^2}{8 (\paraK ^2 + 1)},
\]
{where} the last inequality uses the assumption $2\paraK  \alphaxtau(x) < \betaxtau(x)$. 
The same  condition also implies that $\Fcepstau(x)  =  \frac{1}{2} \alphaxtau(x)^2 + \frac{1}{2} \betaxtau(x)^2 \leq \frac{(4 \paraK ^2 + 1)\betaxtau(x)^2}{8 \paraK ^2 }$. Using $\paraK \ge 1$, these two bounds in turn, and {also} the lower bound on $\Fcepstau(x)$ assumed in the lemma, we obtain 
    \begin{align*}
    \Fcepstau(x) - \Fcepstau(x') > \frac{\betaxtau^2(x)}{8 (\paraK ^2 + 1)} \geq \frac{ \paraK ^2 }{(4 \paraK ^2 + 1)(\paraK ^2 + 1)}\Fcepstau(x) \geq \frac{1}{10 \paraK ^2} \Fcepstau(x)\ge \zeta\, ,
    \end{align*}
    a contradiction to the assumption that $x$ is a $\zeta$-approximate minimizer of $\Fcepstau(x)$.
\end{proof}

We are {now} ready to prove Lemma~\ref{lem:linear-program-upper-theta-delta}.
\begin{proof}[Proof of Lemma~\ref{lem:linear-program-upper-theta-delta}]
    By Lemma~\ref{lem:lower-bound-alpha-linear-system}, as $\Fcepstau(x) \ge \paraBarK ^2 \zeta \geq 10 \paraK ^2 \zeta$,  
    \begin{align*}
    \alphaxtau \geq \frac{\sqrt{\Fcepstau(x)}}{2 \paraK } 
    \ge \frac {\paraBarK  \sqrt{\zeta}} {2\paraK } \cdot
    \geq 32 n \cdot \kappa(\cX_{\oriA}) \cdot\sqrt{\zeta}. 
    \end{align*}
    {Thus,} 
    \begin{align*}
     \varepsilon \alphaxtau
 = \frac {   \alphaxtau} { 8n\cdot\kappa(\cX_{\oriA})}
    \geq 4 \sqrt{\zeta}.
    \end{align*}
    Lemma~\ref{lem:linear-program-theta-upperbound} now gives $\theta(x,  \pi, \sigma) \leq n\sqrt{ \zeta} / 2$, proving (i).

    In addition,  the definition of $\Fcepstau(x)$ implies that  $\|Ax - b\|_2^2 \le 2 \|A\|_1^2 \cdot\Fcepstau(x)$. The assumption
    $\Fcepstau(x) \leq 2 \paraBarK ^2 \zeta$ implies that
    \begin{align*}
    \|Ax - b\|_1 \leq \sqrt{m} \cdot \|Ax - b\|_2 \leq  \sqrt{2m} \|A\|_1 \cdot \sqrt{\Fcepstau(x)} \leq 2 \|A\|_1\cdot \paraBarK \sqrt{m \zeta}\, ,
    \end{align*} 
    proving (ii).
    
    Finally, 
    \begin{align*}
        \|\pi\|_{\infty} &= \frac{\|c\|_{\infty}}{\alphaxtau\|A\|_1^2}\left\|Ax - b\right\|_{\infty}  \leq \frac{2\paraK  \|c\|_{\infty} }{\paraBarK  \sqrt{\zeta}\|A\|_1^2} \|Ax - b\|_{2} \\
        &\leq 4 \paraK  \|c\|_{\infty} / \|A\|_1
        = 4n \sqrt{m} \cdot \kappa(\cX_{\oriA})\cdot \|c\|_{\infty}\, , 
    \end{align*}
    proving (iii).
\end{proof}

\subsection{Convergence speed of R-FGM}
\label{subsed:R-FEM-pf-convergence}

Let us define $B\in\R^{(m+1)\times (n+1)}$, $\tilde u,\tilde b\in\R^{m+1}$ as
\[
B:=\begin{pmatrix} A& 0 \\ \|A\|_1 \ceps^\top  & \|A\|_1 \end{pmatrix}\, ,\quad \tilde u:=\begin{pmatrix}u\\ M \end{pmatrix}\, \quad \tilde b:=\begin{pmatrix}b  \\ \|A_1\|\tau \end{pmatrix}
\]
for sufficiently large $M$. With this notation, minimizing $\Ft{\tau}(x)$ over $0\le x\le u$ can be written in the form 
\begin{equation}\label{eq:embedded}
\begin{aligned}
\min \frac{1}{2\|A\|_1^2} \cdot \left \|B\begin{pmatrix}x\\t\end{pmatrix}-\tilde b\right\|_2^2\\
0\le \begin{pmatrix}x\\t\end{pmatrix}\le \tilde u\, .
\end{aligned}
\end{equation}
We let $\Ftt{\tau}(x,t)$ denote the objective function.
We restate Lemma~\ref{lem:R-FEM-pf-convergence} for convenience.

\RFGMconvergence*

We use the starting point $(x^0, t^0) = (0, \tau)$. The bound follows from Theorem~\ref{thm:r-fgm-convergence-rate}, with the smoothness and quadratic growth bounds as below.
\begin{lemma}\label{lem:linear-program-Lipz}
The function $\Ftt{\tau}(x,t)$   is $(2n+2)$-smooth.
\end{lemma}
\begin{proof}
We can bound the smoothness parameter as $\|B\|_2^2  / \|A\|_1^2\leq (n + 1)\|B\|_1^2  / \|A\|_1^2 \leq (n +1) (\max_i (A_i + \|A\|_1\ceps_i  ))^2   / \|A\|_1^2 \leq 2(n +1) $.
\end{proof}

\begin{lemma}[Quadratic growth] \label{lem:linear-program-QFG} 
\[
\kappa(\cX_{B}) \leq 2(m+1)\|A\|_1 \cdot \bar\kappa^2(\cX_{\oriA}) \cdot\frac{1}{\varepsilon}\, .
\]
Consequently,
the function $\Ftt{\tau}(x,t)$  has $\varepsilon^2 \big/ \big[64 m^4 \cdot \|A\|_1^4\cdot\bar\kappa^4(\cX_{\oriA})\big] $-quadratic growth. 
\end{lemma}

The proof is based on the following  lemma.

\begin{lemma}\label{lem:kappa-bound-two-decomp-2}
For a matrix $A\in\R^{m\times n}$ and a vector $d\in\R^{n}$, let $K=\begin{pmatrix}A\\ d^\top\end{pmatrix}$. Then, every elementary vector in $\EE(K)$ is either an elementary vector in $\EE(A)$, or the sum of two conformal elementary vectors in $\EE(A)$.
Further, if $d\in \mathbb{Z}^n$, $d\neq\bO$, then 
\[
\kappa(K)\le 2(m+1)\cdot\|d\|_\infty\cdot\bar\kappa^2(A)\, .
\]
\end{lemma}
\begin{proof}
Let $z$ be an elementary vector in $\EE(K)$, and consider a generalized circuit-path decomposition of $z$ w.r.t. $A$, 
\[
z=\sum_{k=1}^h  g^k\, ,
\]
where
$h\le n$ and  $g^1,g^2,\ldots,g^h\in \EE({A})$  are elementary vectors that conform to
$z$. Further, for each $i=1,2,\ldots,h-1$, $\supp(g^i)\setminus \cup_{j=i+1}^h \supp(g^j)\neq \emptyset$. 

The first statement follows by showing $h \leq 2$. The proof is by contradiction: suppose $h > 2$, and consider $g^2$ and $g^3$.

First, we observe that $\supp(g^2) \cup \supp(g^3) \subsetneq \supp(z)$, because  $\supp(g^1)\setminus \cup_{j=2}^h \supp(g^j)\neq \emptyset$.
Next, we show that $\pr{c}{g^2} \neq 0$ and $\pr{c}{g^3} \neq 0$. For if one of them equals $0$, for example, $\pr{c}{g^2} = 0$, then as $g^2$ is elementary, $A g^2 = 0 $ also, which implies $z$ is not an elementary vector of $\EE(K)$, as $g_2$ has a strictly smaller support than $z$.

To obtain a contradiction, we consider $g^{23} = \pr{c}{g^2} g^3 - \pr{c}{g^3} g^2$, which is non-zero as $\supp(g^2)\setminus \cup_{j=3}^h \supp(g^j)\neq \emptyset$ and $\pr{c}{g^3} \neq 0$. $g^{23}$ also has a strictly smaller support than $z$, as $\supp(g_2) \cup \supp(g_3) \subsetneq \supp(z)$. In addition,  $\pr{c}{g^{23}} = 0$ and $A g^{23} = 0$.  This implies $z$ is not an elementary vector of $\EE(K)$, which provides a contradiction.

\medskip

Let us now turn to the second statement: assume that $d\in\Z^n$. Take an elementary vector $z$ in $\EE(K)$ such that $\kappa(K)=|z_i/z_j|$ for some $i,j\in \supp(z)$.
If  $z\in\EE(A)$, then $\kappa(K)\le\kappa(A)$, and hence the bound follows.

Otherwise,  $z$ is the sum of two elementary vectors in $\EE(A)$. After appropriately scaling $z$, it can be written in the form 
$z=\lambda_1 g^1+\lambda_2 g^2$ with 
$g^1,g^2\in\bar\EE(A)$, i.e., they are integer vectors such that the largest common divisor of their entries is 1. Also, $\lambda_1,\lambda_2\neq 0$. Further, we must have $0=\pr{d}{z}=\lambda_1\pr{d}{g^1}+\lambda_2\pr{d}{g^2}$. After possibly another scaling of $z$, we get $\lambda_1=\pr{d}{g^2}$ and $\lambda_2=-\pr{d}{g^1}$, that is, 
\[
z=\pr{d}{g^2} g^1-\pr{d}{g^1} g^2\, .
\]
By {the} definition of $\bar\kappa(A)$, $\|g^1\|_\infty, \|g^2\|_\infty\le\bar\kappa(A)$. By the integrality of $g^1$ and $g^2$, 
\[
1\le |\pr{d}{g^2}|, |\pr{d}{g^1}|\le (m+1)\cdot  \|d\|_\infty\cdot\bar\kappa(A)\, .
\]
Thus, all nonzero entries of $z$ have $1\le |z_i|\le 2(m+1)\cdot  \|d\|_\infty\cdot\bar\kappa(A)\cdot\kappa(A)$ , implying the claim since $\kappa(A)\le\bar\kappa(A)$. 
\end{proof}

\begin{proof}[Proof of Lemma~\ref{lem:linear-program-QFG}]
The bound on the quadratic growth parameter follows from the circuit imbalance bound: By Lemma~\ref{lem:qfg-ne}, $\Ftt{\tau}(x,t)$ has $1 / \theta_{2,2}^2(B)$-quadratic growth, and, by Lemma~\ref{lem:hoffman-kappa}, $\theta_{2,2}(B) \le (m+1)\cdot\kappa(\cX_{B})$.

Let us now show the circuit imbalance bound.
Recall that $\cX_{\oriA} = \ker\begin{pmatrix} A&-I_m\end{pmatrix}$. Since arbitrarily scaling the rows of a matrix does not change the kernel and thus does not affect the circuit imbalances, we can write  $\cX_{B} =\ker\begin{pmatrix}B & -I_m\end{pmatrix}=\ker(H)$ for 
\[H\coloneqq\begin{pmatrix} A & 0&-I_m&0 \\ \frac{1}{\varepsilon}\ceps^\top   &   \frac{1}{\varepsilon} &0 &- {\frac{1}{\|A\|_1\cdot \varepsilon}} \end{pmatrix}\, .\] 
Let $H'$ be the matrix arising from $H$ by deleting the last column, and let us also define $D=\begin{pmatrix} A & 0&-I_m\end{pmatrix}$.

Clearly, $ \bar\kappa(\cX_{\oriA}) = \bar\kappa(D)$. 
Recall from  \eqref{eq:epsilon-def} that $1/\varepsilon$ is defined to be an integer. Hence, 
Lemma~\ref{lem:kappa-bound-two-decomp-2} is applicable to the matrix $H'$. Note that the last row of this matrix has $\ell_\infty$ norm $1/\varepsilon$, since $\|\ceps\|_\infty=1$. Therefore, $\kappa(H')\le 2(m+1)\cdot \bar\kappa(\cX_A) \cdot \frac{1}{\varepsilon}$.

We show that $\kappa(H)\le \|A\|_1\cdot\kappa(H')$; this implies the statement. Indeed, $H$ arises from $H'$ by duplicating one of its columns and scaling it by $1/\|A\|_1$. Duplicating columns does not affect the circuit imbalance, whereas multiplying a column by any constant ${1/}\alpha$ may increase it by at most a factor $\alpha$. This completes the proof.
\end{proof}

\section{Analysis of dual certificates}\label{sec:cert-proof}

We now present the proof of Lemma~\ref{lem:delta-cert} and Theorem~\ref{thm:dual-check} on $\delta$-certificates.
\begin{proof}[Proof of Lemma~\ref{lem:delta-cert}]
{\bf Part (i)}~ Let $(\pi,w^-,w^+)$ denote the $\delta$-certificate. This is a feasible solution to \ref{eq:LP-dual}, and therefore $\optv(A,b,c,u)\ge \pr{b}{\pi}-\pr{u}{w^+}$. 
Thus, using the properties of $\delta$-certificates, we get the bound
\[
\begin{aligned}
\pr{c}{x}-\optv(A,b,c,u)&\le \pr{c}{x}-\pr{b}{\pi}+\pr{u}{w^+}\\
&=\pr{c-A^\top \pi+w^+}{x}+\pr{{A} x-b}{\pi}+\pr{u-x}{w^+}\\
&=\pr{w^-}{x}+\pr{w^+}{u-x}+\pr{{A} x-b}{\pi}\\
&\le (4n+2)\delta\|c\|_\infty\, .
\end{aligned}
\]
\paragraph{Part (ii)} Let $\bar b=Ax$. By Lemma~\ref{lem:optval-prox}, and the assumption on $\delta^\feas$, it follows that 
\[
|\optv(A,b,c,u)-\optv(A,\bar b,c,u)|\le \kappa(\cX_A)\cdot\|c\|_1\cdot \|b-\bar b\|_1\le n\cdot \kappa(\cX_A)\cdot\|c\|_\infty\cdot \|A\|_1\cdot\delta^\feas \le \delta^\opt\cdot\|c\|_\infty\, ,
\]

Thus, $x$ is $2\delta^\opt$-optimal {for}  $\LP(A,\bar b,c,u)$. Let us select an optimal dual solution $(\bar \pi,\bar w^-,\bar w^+)$ to $\Dual(A,\bar b,c,u)$. Thus, $\pr{\bar b}{\bar \pi}-\pr{u}{\bar w^+}=\optv(A,\bar b,c,u)$. As in part (i), we get  
\[
\begin{aligned}
2\delta^\opt\|c\|_\infty &\ge \pr{c}{x}-\optv(A,{\bbar},c,u)= \pr{c}{x}-\pr{{\bbar}}{\bar \pi}+\pr{u}{\bar w^+}\\
&=\pr{\bar w^-}{x}+\pr{\bar w^+}{u-x}+\pr{A x-{\bbar}}{\bar \pi}\, .
\end{aligned}
\]
Since all terms here are nonnegative, we get that 
$(\bar \pi,\bar w^-,\bar w^+)$ is a feasible solution to the LP
\begin{equation}\label{eq:dual-compl}
\begin{aligned}
A^\top \pi&+w^--w^+= c\,\\
0\le w^-_i&\le \frac{2\delta^\opt\|c\|_\infty}{x_i}\, ,\quad \forall i\in [n]\\
0\le w^+_i&\le \frac{2\delta^\opt\|c\|_\infty}{u_i-x_i}\, ,\quad \forall i\in [n]\, .
\end{aligned}
\end{equation}
We now show that \eqref{eq:dual-compl} has a feasible solution $(\pi,w^-,w^+)$ with 
\begin{equation}\label{eq:pi-bound}
\|\pi\|_\infty\le \kappa(\cX_A)\cdot \|c\|_1\, . 
\end{equation}
This implies that $(\pi,w^-,w^+)$  also satisfies requirement {(iii)} in the definition of $\delta^\opt$-certificates (Definition~\ref{def:dual-cert}), since $$\kappa(\cX_A)\cdot \|c\|_1 \le \|c\|_1\cdot\delta^\opt/[\delta^\feas\cdot n\cdot\|A\|_1] \le (\|c\|_1/n) \cdot \delta^\opt/\|Ax-b\|_1 \le
\delta^\opt\|c\|_\infty/\|Ax-b\|_1,$$
{where the first inequality uses the} assumption {that} $\delta^\feas \cdot n\cdot\kappa(\cX_A) \cdot \|A\|_1 \le \delta^\opt$, {and the second inequality uses the fact that $x$ is $\deltafeas$}.

We now show the existence of such a solution to \eqref{eq:dual-compl}. 
 Let $c'\coloneqq\max\{\bO,c\}$ and $c''\coloneqq\max\{\bO,-c\}$. Thus, $(\bO,c',c'')$ satisfies all inequalities in  \eqref{eq:dual-compl} except the upper bounds. Now, let $(\pi,w^-,w^+)$  be a feasible solution to  \eqref{eq:dual-compl} such that the distance
 $\|(\bO,c',c'')-(\pi,w^-,w^+)\|_2$ is minimal.

Let $\mathcal{Y}:=\ker(A^\top | I_n| -I_n)$. By Lemma~\ref{lem:cX-orth}, and noting that duplicating columns does not affect the circuit imbalances, we see that $\kappa(\mathcal{Y})=\kappa(\cX_A)$. 

Note that $(\bO,c',c'')-(\pi,w^-,w^+)\in\mathcal{Y}$. Let $\sum_{k=1}^h g^k$ be a conformal circuit decomposition of the difference of these vectors, where $g^k\in \R^{m\times n\times n}$. By the choice of $(\pi,w^-,w^+)$, the support of each $g^k$ must contain at least one coordinate in the second block with $c'_i>\frac{2\delta^\opt\|c\|_\infty}{x_i}$, or in the third block with $c''_j>\frac{2\delta^\opt\|c\|_\infty}{u_j-x_j}$, and the corresponding component of $g^k$ must be negative.

By the conformity of the circuit decomposition, and the definition of $\kappa(\mathcal{Y})$, it follows that $\|\pi\|_\infty\le \kappa(\mathcal{Y})\cdot (\|c'\|_1+\|c''\|_1)=\kappa(\cX_A)\cdot\|c\|_1$; thus, \eqref{eq:pi-bound} holds. Therefore, $(\pi,w^-,w^+)$ is a $\delta^\opt$-certificate for $x$.
\end{proof}

\begin{proof}[Proof of Theorem~\ref{thm:dual-check}]
Similarly to the feasibility algorithm in Theorem~\ref{thm:feas}, we use \RFGM{} on a convex quadratic minimization problem. Namely, we consider the problem
\[
\begin{aligned}
\min &\frac{1}{2}\left\|A^\top \pi +w^--w^+-c\right\|^2\, \\
0&\le w^-_i\le \frac{2\delta^\opt\|c\|_\infty}{x_i}\, ,\quad \forall i\in [n]\, ,\\
0&\le w^+_i\le \frac{2\delta^\opt\|c\|_\infty}{u_i-x_i}\, ,\quad \forall i\in [n]\, ,\\
-\frac{2\delta^\opt\|c\|_\infty}{\|Ax-b\|_1}&\le \pi_i\le \frac{2\delta^\opt\|c\|_\infty}{\|Ax-b\|_1}\, , \quad \forall i\in [n]\, .
\end{aligned}
\]
Lemma~\ref{lem:delta-cert}(ii) guarantees the existence of a $\delta^\opt$-certificate, and {using \Cref{def:dual-cert}(i), we can deduce that} the optimal value for this program is 0.

We run \RFGM{}, starting from the all zero solution to find an $\varepsilon$-approximate solution, where
$\varepsilon:=\frac{1}{2} \cdot \left(\frac{2\delta^\opt\|c\|_\infty}{\|u\|_1}\right)^2$. Note that $\mu\coloneqq\sqrt{2\varepsilon}$ is smaller than any of the bounds in the box constraints above. Hence, we can modify the solution to $( \pi, \tilde w^-,\tilde w^+)$ such that $A^\top   \pi+\tilde w^--\tilde w^+=c$, and this solution violates the box constraints by at most a factor 2. Thus, $(\pi, \tilde w^-,\tilde w^+)$ is a $2\delta^\opt$-certificate.

By Lemma~\ref{lem:qfg-ne}, the quadratic growth parameter is $1 / \theta_{2,2}^2(A^\top|I_n|-I_n)$, which is {larger}  than $1 / m^2 \kappa^2(\cX_A)$. This bound follows by Lemma~\ref{lem:cX-orth} and the fact that duplicating columns does not affect the circuit imbalances. In addition, by Lemma~\ref{lem:Lips}, the smoothness parameter is $(\|A\|_2 + 1)^2$ as 
\[
\|(A^\top|I_n|-I_n)\|_2 = \max_{p, q, r} \frac{\sqrt{\|A^\top p\|_2^2 + \|q\|_2^2 + \|r\|_2^2}}{\sqrt{\| p\|_2^2 + \|q\|_2^2 + \|r\|_2^2}}\leq \max_{p, q, r}  \frac{\sqrt{\|A^\top p\|_2^2}}{\sqrt{\| p\|_2^2 }} + \frac{\sqrt{ \|q\|_2^2 + \|r\|_2^2}}{\sqrt{ \|q\|_2^2 + \|r\|_2^2}} \leq  \|A\|_2 + 1.
\]
Given this, by {\Cref{thm:r-fgm-convergence-rate}}, the total number of iterations of \RFGM{} is at most $O\big(m \|A\|_2 \cdot\kappa(\cX_A)\cdot \log(n \|u\|_1 / \delta^\opt)\big)$.
\end{proof}

\section{Hoffman constant example for the self-dual embedding}\label{sec::examples}
We show that for the self-dual embedding \eqref{eq:self-dual}, the Hoffman constant of the corresponding matrix can be unbounded.

\begin{lemma}
    Let $A\in\R^{m\times n}$, $b\in\R^m$, $u\in \R^n$ with $\cP_{A,b,u}\neq\emptyset$ being at least 2-dimensional. Then, the Hoffman constant $\theta_{2,2}$ corresponding to the system
\[
\mathcal{S}_c = \left \{ (x, \pi, w^-, w^+) ~\Big|\, \, ~ x \in \cP_{A,b,u} \, ; A^\top \pi + w^--w^+=c\, ;  \pr{c}{x}-\pr{b}{\pi}+\pr{u}{w^+}=0\, ; w^-,w^+\ge 0\,\right \}
\]
can be arbitrarily large. 
\end{lemma}
\begin{proof}
Let us denote $\cP=\cP_{A,b,u}$, and pick any $\varepsilon>0$. We show that the Hoffman-constant can be at least $1/\varepsilon$. 

    Consider any facet $\mathcal{F}$ of $\cP_{A,b}$, corresponding to the linear equation $\pr{r}{x} + v = 0$ for some $r \in \mathbb{R}^n$ and $v \in \mathbb{R}$. We define $c\approx r$ as a perturbation of $r$with $\|c-r\|_2\le\varepsilon$, such that  $\max \pr{c}{x} ~\text{s.t.}~ x \in \cP$ has a unique optimal solution $\bar x\in \mathcal{F}$. Let $x'\in \mathcal{F}\setminus\{x'\}$ be another extreme point. Let $(\pi, w^+, w^-)$ be an optimal solution to the dual of  $\max \pr{c}{x} ~\text{s.t.}~ x \in \cP$. Then, for the primal-dual pair $(x', \pi, w^+, w^-)$,
    \begin{enumerate}
    \item $x'\in \cP$ ;
    \item  $(\pi, w^+, w^-)$ is feasible: $A^\top \pi + w^--w^+=c$ and $w^+, w^- \geq 0$;
    \item the optimality gap is tiny but strictly positive, $\pr{c}{x'}-\pr{b}{\pi}+\pr{u}{w^+} \leq \pr{c}{x' - \bar{x}} \leq \|c-r\|_2\cdot \|x' - \bar{x}\|_2\le\varepsilon \|x' - \bar{x}\|_2$ and $\pr{c}{x'}-\pr{b}{\pi}+\pr{u}{w^+} > 0$;
    \item the distance to the nearest point in $\mathcal{S}_c$ is $\|x' - \bar{x}\|_2$.
    \end{enumerate}
    In this case, the Hoffman constant $\theta_{2,2}$ is at least the distance to the nearest point in $\mathcal{S}_c$ divided by the optimality gap, which is at least ${1}/{\varepsilon}$.
\end{proof}

\bibliographystyle{abbrv}
\bibliography{references.bib}

\end{document}